\documentclass[reqno]{amsart}
\usepackage{amssymb}
\usepackage{graphicx}

\usepackage[usenames, dvipsnames]{color}
\usepackage{verbatim}
\usepackage{mathrsfs}
\usepackage{bm}
\usepackage{cite}

\numberwithin{equation}{section}

\newtheorem{theorem}{Theorem}[section]
\newtheorem{corollary}[theorem]{Corollary}
\newtheorem{lemma}[theorem]{Lemma}
\newtheorem{prop}[theorem]{Proposition}
\newtheorem*{thmA}{Theorem A}

\theoremstyle{definition}
\newtheorem{remark}[theorem]{Remark}

\theoremstyle{definition}

\theoremstyle{definition}

\makeatletter
\def\dashint{\operatorname%
{\,\,\text{\bf-}\kern-.98em\DOTSI\intop\ilimits@\!\!}}
\makeatother

\def\\det{\text{det}}

\def\.5{\frac{1}{2}}

\newcommand{\RN}[1]{%
  \textup{\uppercase\expandafter{\romannumeral#1}}%
}

\renewcommand{\epsilon}{\varepsilon}

\newcounter{marnote}

\begin{document}

% ------------------------------------------------------------------------

\title[Asymptotics for the perfect conductivity problem]{Asymptotics for the electric field concentration in the perfect conductivity problem}

\author[H.G. Li]{Haigang Li}
\address[H.G. Li]{School of Mathematical Sciences, Beijing Normal University, Laboratory of Mathematics and Complex Systems, Ministry of Education, Beijing 100875, China. }
\email{hgli@bnu.edu.cn}
\thanks{H.G. Li was partially supported by  NSFC (11571042, 11631002, 11971061) and BJNSF (1202013)}

\date{\today} % delete this line to display the current date

\maketitle

% ------------------------------------------------------------------------

\begin{abstract}
In the perfect conductivity problem of composite material, the electric field concentrates in a narrow region in between two inclusions and always becomes arbitrarily large when the distance between inclusions tends to zero. To characterize such singular behavior, we capture the leading term of the gradient and reveal that the blow-up rates are determined by their relative convexity of the two adjacent inclusions. On the other hand, a blow-up factor, which is a linear functional of boundary data, is found to determine the blow-up will occur or not.  
\end{abstract}

\section{Introduction}

\subsection{Background}
In composite materials, the inclusion are frequently located very closely and even touching. Especially, in high-contrast fiber-reinforced composites, it is a common phenomenon that high concentration of extreme electric field or stress field occurs in the narrow regions between two adjacent inclusions. The purpose of this paper is to investigate the asymptotic behavior of the electric field in the perfect conductivity problem when the distance between inclusions tends to zero. The conductivity problem can be modeled by the following boundary problem of the scalar equation with piecewise constant coefficients
\begin{equation}\label{equk}
\begin{cases}
\mathrm{div}\Big(a_{k}(x)\nabla{u}_{k}\Big)=0&\mbox{in}~D,\\
u_{k}=\varphi(x)&\mbox{on}~\partial{D},
\end{cases}
\end{equation}
where $D$ is a bounded open set in $\mathbb R^{n}$, $n\geq2$, including two inclusions $D_{1}$ and $D_{2}$ with $\varepsilon$ apart, $\varphi\in{C}^{2}(\partial{D})$ is given, and
$$a_{k}(x)=
\begin{cases}
k\in[0,1)\cup(1,\infty]&\mbox{in}~D_{1}\cup{D}_{2},\\
1&\mbox{in}~\Omega=D\setminus\overline{D_{1}\cup{D}_{2}}.
\end{cases}
$$
The gradient of the potential $u$ represents the electric field, $a_{k}(x)$ is the conductivity, which is a constant on the fibers, and a different constant on the matrix. When the conductivity of inclusions degenerate into infinity, we call it as {\em{the perfect conductivity problem}}. It is important from a practical point of view to know whether $|Du|$ can be arbitrarily large as the inclusions get closer to each other. Motivated by the celebrated work of Babu\v{s}ka, Andersson, Smith, and Levin \cite{basl} where they numerically analyzed the initiation and growth of damage in composite materials, in which the inclusions are frequently spaced very closely and even touching, there have been many important works on the gradient estimates for solutions of elliptic and parabolic equations and systems arising from composite materials; see, for instance, \cite{bv,d,dl,fknn,jlx,kim,ll,llby,ln,lv} and the references therein.

When $k$ is away from 0 and $\infty$, the gradient of the solution of \eqref{equk}, $|\nabla{u}_{k}|$, is bounded independently of the distance $\varepsilon$. Bonnetier and Vogelius \cite{bv} first obtained the $W^{1,\infty}$ estimate of $u_{k}$ for two touching disks $D_{1}$ and $D_{2}$ in dimension two, which improved a classical regularity result due to De Giorgi and Nash \cite{DeGi,nash}, which asserts that the $H^{1}$ weak solution is in the H\"{o}lder class for $L^{\infty}$ coefficients. Of course, the bound in \cite{bv} depends on the value of $k$. Li and Vogelius  \cite{lv} and Li and Nirenberg \cite{ln} extended such boundedness result to general divergence form second order elliptic equations and systems with piecewise H\"older continuous coefficients, and they proved that $|\nabla{u}_{k}|$ remains bounded when $\varepsilon$ tends to zero. 

Actually, this is a bi-parameter problem, including two independent parameters: the contrast $k$ and the distance $\varepsilon$. In order to study the role of $\varepsilon$ played in such kind of concentration phenomenon, we consider another limit case with $k =+\infty$, the perfect conductivity problem:
\begin{equation}\label{equinfty2}
\begin{cases}
\Delta{u}=0&\mbox{in}~\Omega,\\
u=C_{i}&\mbox{on}~\overline{D}_{i},~i=1,2,\\
\int_{\partial{D}_{i}}\frac{\partial{u}~~}{\partial\nu^{-}}=0&i=1,2,\\
u=\varphi(x)&\mbox{on}~\partial{D},
\end{cases}
\end{equation}
where $C_{1}$ and $C_{2}$ are some constants to be uniquely determined, $\varphi\in{C}^{2}(\partial{D})$, and for $x\in\partial{D}_{i}$
$$\frac{\partial{u}~~}{\partial\nu^{-}}(x):=\lim_{t\rightarrow0^{+}}\frac{u(x)-u(x+t\nu)}{t}.$$
Here and throughout this paper $\nu$ is the outward unit normal to
the domain. It has been proved that the generic blow-up rate of $|\nabla{u}|$ is $\varepsilon^{-1/2}$ in two dimensions \cite{akl,aklll,y1,y2,akllz,ly}, $(\varepsilon|\log\varepsilon|)^{-1}$ in three dimensions \cite{bly1,ly,kly2}, and $\varepsilon^{-1}$ in higher dimensions \cite{bly1}. Similar results for Lam\'{e} system with partially inifinite coefficients were established in \cite{bjl,bll,bll2}, for $p$-Laplace equation in dimension two in \cite{gn}. More earlier work for the blow-up rate of a special solution with two identical circular inclusions was shown to be $\varepsilon^{-1/2}$, see \cite{bc,m,keller}. 

Bao, Li and Yin \cite{bly1} introduced a linear functional $Q_{\varepsilon}[\varphi]$ and obtained the optimal bounds
$$
\frac{\rho_{n}(\varepsilon)|Q_{\varepsilon}[\varphi]|}{C\varepsilon}\leq\|\nabla{u}\|_{L^{\infty}(\Omega)}\leq\frac{C\rho_{n}(\varepsilon)|Q_{\varepsilon}[\varphi]|}{\varepsilon}+C\|\varphi\|_{C^{2}(\partial{D})},
$$
where
\begin{equation}\label{def_rhon}
\rho_{n}(\varepsilon)=
\begin{cases}
\sqrt{\varepsilon}~~&\mbox{for}~n=2;\\
|\log\varepsilon|^{-1}&\mbox{for}~n=3.
\end{cases}
\end{equation}
If $|Q_{\varepsilon}[\varphi]|$ has a strictly positive lower bound independent of $\varepsilon$, then these inequality will show these blow-up rates are optimal. From the view of practical application in engineering,  it is desirous and more important to know how to capture the leading term of such blow-up. Recently a better understanding of the stress concentration has been obtained in \cite{kly,ackly} that an asymptotic behavior of $\nabla{u}$ has been characterized by the singular function $q_{\varepsilon}$ associated with $D_{1}$ and $D_{2}$ in dimension two, and the asymptotic behavior of the stress concentration factor is also considered in \cite{kly}. Ammari, Ciraolo, Kang, Lee, Yun \cite{ackly} extend the result in \cite{kly} to the case that inclusions $D_{1},D_{2}$ are strictly convex simply connected domain in $\mathbb{R}^{2}$. For two adjacent spherical inclusions in $\mathbb{R}^{3}$ was studied by Kang, Lim and Yun \cite{kly2} and Li, Wang and Xu \cite{lwx}. Bonnetier and Triki \cite{bt} derived the asymptotics of the eigenvalues of the Poincar\'e variational problem as the distance between the inclusions tends to zero. Here it is also worth mentioning that Berlyand, Gorb and Novikov \cite{bgn} used a network approximation to estimate the global stress in a composite with densely packed spherical inclusions. 

In this paper, we give an essentially complete description of the gradient asymptotic expansion for arbitrary convex inclusions in all dimensions. The method is quite different with that used in \cite{ackly,kly,kly2}. Motivated by the decomposition in \cite{lx} for the boundary estimates, we here decompose the solution $u$ of \eqref{equinfty2} as follows
\begin{equation}\label{decomposition_u}
u(x)=(C_{1}-C_{2})v_{1}(x)+v_{b}(x),\quad\mbox{in}~~\Omega,
\end{equation}
where $v_{1}$ and $v_{b}$ are, respectively, the solutions of
\begin{equation}\label{equv1}
\begin{cases}
\Delta{v}_{1}=0&\mbox{in}~\Omega,\\
v_{1}=1&\mbox{on}~\partial{D}_{1},\\
v_{1}=0&\mbox{on}~\partial{D}_{2}\cup\partial{D},\end{cases}
\quad\mbox{and}\quad
\begin{cases}
\Delta{v}_{b}=0&\mbox{in}~\Omega,\\
v_{b}=C_{2}&\mbox{on}~\partial{D}_{1}\cup\partial{D}_{2},\\
v_{b}=\varphi(x)&\mbox{on}~\partial{D}.
\end{cases}
\end{equation}
It follows from \eqref{decomposition_u} that
\begin{equation}\label{decom_u2}
\nabla{u}=(C_{1}-C_{2})\nabla{v}_{1}+\nabla{v}_{b}.
\end{equation}
This decomposition comes with a significant advantage: $\nabla{v}_{1}$ is a singular part with an intuitive singularity $\varepsilon^{-1}$, while $\nabla{v}_{b}$ is a bounded part. Thus, the main reason to cause the difference of the rate of the blow-up lies in the term $|C_{1}-C_{2}|$. It turns out that it depends on the dimension $n$ and the geometry of the inclusions. On the other hand, the bounded term $\nabla{v}_{b}$ is also important, because it is closely related to the blow-up factor $\mathcal{B}_{0}[\varphi]$, which decides whether the blow-up will occur or not. For more details, see Proposition \ref{prop2} below.

\subsection{Notations and Main Results}
We now proceed to state the main results of this paper. To do so we need to make our notation and assumptions more precise. We use $x=(x',x_{n})$ to denote a point in $\mathbb{R}^{n}$, $n\geq2$, $x'=(x_{1},x_{2},\cdots,x_{n-1})$. We assume that $\partial{D}$ is of $C^{2,\alpha}$, $0<\alpha<1$. Let $D_{1}^{0}$ and $D_{2}^{0}$ be a pair of (touching) convex subdomains of $D$  and far away from $\partial D$, such that
$$D_{1}^{0}\subset\{(x',x_{n})\in\mathbb R^{n}| x_{n}>0\},\quad D_{2}^{0}\subset\{(x',x_{n})\in\mathbb R^{n}| x_{n}<0\},$$
with $x_{n}=0$ as their common tangent plane, and
$$\partial D_{1}^{0}\cap\partial D_{2}^{0}=\{(0',0)\},\qquad \mathrm{dist}(D^{0}_{1}\cup{D}^{0}_{2},\partial{D})>\kappa_{0},$$
where $\kappa_{0}>1$ is a constant. We further assume that the $C^{2,\alpha}$ norms of
$\partial{D}_{i}$ $(i=1,2)$ are bounded by some constants. By translating $D_{1}^{0}$ by a positive number $\varepsilon$ along $x_{n}$-axis, while $D_{2}^{0}$ is fixed, we obtain $D_{1}^{\varepsilon}$, that is,
$$D_{1}^{\varepsilon}:=D_{1}^{0}+(0',\varepsilon).$$
When there is no possibility of confusion, we drop the superscripts and denote
$$D_{1}:=D_{1}^{\varepsilon},\quad D_{2}:=D_{2}^{0}, \quad\mbox{and}~~\Omega:=D\setminus\overline{D_{1}\cup{D}_{2}}.$$

We may assume that the points $P_{1}\in\partial{D}_{1}$ and $P_{2}\in\partial{D}_{2}$ satisfy
\begin{equation*}\label{P1P2}
P_{1}=\left(0',\varepsilon\right),\quad\,P_{2}=\left(0',0\right).
\end{equation*}
Fix a small universal constant $R_{0}<1$ such that the portions of
 $\partial{D}_{i}$  near $P_{i}$  can be parameterized by $(x',\varepsilon+h_{1}(x'))$ and $(x',h_{2}(x'))$,
respectively, that is,
$$x_{n}=\varepsilon+h_{1}(x'),~\mbox{and}~\,x_{n}=h_{2}(x'), \quad\mbox{for}~~x'\in{B}'_{2R_{0}}:=\left\{x'\in\mathbb{R}^{n-1}~\Big|~|x'|<2R_{0}\right\}.$$
Moreover, by
the convexity assumptions on $\partial{D}_{i}$, we further assume that functions
  $h_{1}$ and $h_{2}$  satisfy
\begin{equation}\label{h1h200}
\varepsilon+h_{1}(x')>h_{2}(x'),\quad\mbox{for}~~|x'|<2R_{0},
\end{equation}
\begin{equation}\label{h1h20}
h_{1}(0')=h_{2}(0')=0,\quad\nabla_{x'}h_{1}(0')=\nabla_{x'}h_{2}(0')=0,
\end{equation}
and for some constant $\kappa_{1}>0$, and for any $\xi\in\mathbb{R}^{n-1}\setminus\{0'\}$, 
\begin{equation}\label{h1h21}
\xi^{T}\nabla^{2}_{x'}h_{1}(0')\xi\geq\kappa_{1}|\xi|^{2}>0,~\quad\,\xi^{T}\nabla^{2}_{x'}h_{2}(0')\xi\leq-\kappa_{1}|\xi|^{2}<0.
\end{equation}
and
\begin{equation}\label{h1h2}
\|h_{1}\|_{C^{3,1}(B'_{2R_{0}})}+\|h_{2}\|_{C^{3,1}(B'_{2R_{0}})}\leq{C}.
\end{equation}
More generally, after a rotation of the coordinates if necessary, we assume that
\begin{equation}\label{h1h23}
(h_{1}-h_{2})(x')=\sum_{j=1}^{n-1}\frac{\lambda_{j}}{2}x_{j}^{2}+O(|x'|^{2+\alpha}),\quad\,|x'|\leq\,2R_{0},
\end{equation}
where $\mathrm{diag}(\lambda_{1},\cdots,\lambda_{n-1})=\nabla_{x'}^{2}(h_{1}-h_{2})(0')$. For $0\leq\,r\leq\,2R_{0}$, let
$$ \Omega_r:=\left\{(x',x_{n})\in \mathbb{R}^{n}~\big|~h_{2}(x')<x_{n}<\varepsilon+h_{1}(x'),~|x'|<r\right\}.$$

We introduce an auxiliary function $\bar{v}_{1}\in{C}^{2,\alpha}(\mathbb{R}^{n})$, such that $\bar{v}_{1}=1$ on $\partial{D}_{1}$, $\bar{v}_{1}=0$ on $\partial{D}_{2}\cup\partial{D}$,
\begin{align}\label{ubar}
\bar{v}_{1}(x)
=\frac{x_{n}-h_{2}(x')}{\varepsilon+(h_{1}-h_{2})(x')},\quad\mbox{in}~~\Omega_{2R_{0}},
\end{align}
and
\begin{equation}\label{nablau_bar_outside}
\|\bar{v}_{1}\|_{C^{2,\alpha}(\mathbb{R}^{n}\setminus \Omega_{R_{0}})}\leq\,C.
\end{equation}
In view of \eqref{h1h20}--\eqref{h1h2}, a direct calculation gives
\begin{equation}\label{nablau_bar}
\begin{split}
\left|\partial_{x_{j}}\bar{v}_{1}(x)\right|&\leq\frac{C|x'|}{\varepsilon+(h_{1}-h_{2})(x')},~~j=1,2,\cdots,n-1,\\
\quad
\partial_{x_{n}}\bar{v}_{1}(x)&=\frac{1}{\varepsilon+(h_{1}-h_{2})(x')},
\end{split}\qquad~x\in\Omega_{2R_{0}}.
\end{equation}
Here and throughout this paper, unless otherwise stated, $C$ denotes a constant, whose values may vary from line to line, depending only on $n,\kappa_{0},\kappa_{1}$, $\|\partial{\Omega}\|_{C^{2,\alpha}}$, $\|\partial{D}_{1}\|_{C^{2,\alpha}}$ and $\|\partial{D}_{2}\|_{C^{2,\alpha}}$, but not on $\varepsilon$. Also, we call a constant having such dependence a
{\it universal constant}.

Consider the following limit problem
\begin{equation}\label{equ_u0}
\begin{cases}
\Delta{u}_{0}=0&\mbox{in}~\Omega^{0}:=D\setminus\overline{D_{1}^{0}\cup{D}_{2}^{0}},\\
u_{0}=C_{0}&\mbox{on}~\overline{D_{1}^{0}\cup{D}_{2}^{0}},\\
\int_{\partial{D}_{1}^{0}}\frac{\partial{u}_{0}}{\partial\nu^{-}}+\int_{\partial{D}_{2}^{0}}\frac{\partial{u}_{0}}{\partial\nu^{-}}=0&\\
u_{0}=\varphi(x)&\mbox{on}~\partial{D}.
\end{cases}
\end{equation}
It will be shown later that $u_{0}$ is the limit of $u$. We use $u_{0}$ to define a linear functional of $\varphi$, which determines whether $\nabla{u}$ blows up or not,
\begin{equation}\label{Qphi0}
\mathcal{B}_{0}[\varphi]:=-\int_{\partial{D}_{1}^{0}}\frac{\partial{u}_{0}}{\partial\nu^{-}}=\int_{\partial{D}_{2}^{0}}\frac{\partial{u}_{0}}{\partial\nu^{-}}.
\end{equation}
This factor was first introduced by Gorb and Novikov in \cite{gn} for $p$-Laplace equation, denoted by $\mathcal{R}_{0}$ . It turns out there that $\mathcal{R}_{0}$ is the key characteristic parameter of the $W^{1,\infty}$ blow-up of $u$, see also \cite{g}.

In the following, we use $O(1)$ to denote some quantity satisfying $|O(1)|\leq\,C$, for some constant $C$ independent of $\varepsilon$.
We have the asymptotic expression of $\nabla{u}$ in the narrow region between $D_{1}$ and $D_{2}$ as follows:

\begin{theorem}\label{thm1}
For $n=2,3$, let $D$, $D_{1}$, $D_{2}$ be defined as the
above and satisfy \eqref{h1h200}-\eqref{h1h23}, $\varphi\in{C}^{2}(\partial{D})$. Assume that
$u\in{H}^{1}(D)\cap{C}^{1}(\overline{\Omega})$
is the solution to \eqref{equinfty2}. Then for $\varphi$ such that $\mathcal{B}_{0}[\varphi]\neq0$, we have

(i) for $n=2$,
\begin{equation}\label{equ_thm1}
\nabla{u}=\frac{\mathcal{B}_{0}[\varphi]\sqrt{\varepsilon}}{\kappa_{2}}\nabla\bar{v}_{1}+O(1)\|\varphi\|_{C^{2}(\partial{D})},\quad\,\mbox{in}~~\Omega_{R_{0}};
\end{equation}

(ii) for $n=3$,
\begin{equation}\label{equ2_thm1}
\nabla{u}=\frac{\mathcal{B}_{0}[\varphi]}{\kappa_{3}|\log\varepsilon|}\Big(1+O\big(|\log\varepsilon|^{-1}\big)\Big)
\nabla\bar{v}_{1}+O(1)\|\varphi\|_{C^{2}(\partial{D})},\quad\,\mbox{in}~~\Omega_{R_{0}},
\end{equation}
where
\begin{equation}\label{kappa}
\kappa_{n}:=
\begin{cases}
\frac{\sqrt{2}\pi}{\sqrt{\lambda_{1}}}&n=2,\\
\frac{\pi}{\sqrt{\lambda_{1}\lambda_{2}}}&n=3,
\end{cases}
\end{equation}
$\lambda_{1}$ (or $\lambda_{1}$ and $\lambda_{2}$) is the relatively principal curvature of $\partial{D}_{1}$ and $\partial{D}_{2}$, defined in \eqref{h1h23}.
\end{theorem}

\begin{remark}
We would like to point out that from \eqref{nablau_bar} $\nabla\bar{v}_{1}$ is explicit. So the singularity of $\nabla u$ in the narrow region $\Omega_{R}$ can be calculated provided $\mathcal{B}_{0}[\varphi]$ is known for a given $\varphi$. The computation of $\mathcal{B}_{0}[\varphi]$ is an interesting numerical problem, because there is no singularity in $\nabla u_{0}$. We leave it to the interested readers.
\end{remark}

\begin{remark}
This blow-up factor $\mathcal{B}_{0}[\varphi]$ is more natural than $Q_{\epsilon}[\varphi]$ defined in \cite{bly1}, and it is much easier to check whether or not it equals zero, since $\nabla{u}_{0}$ is regular, namely, always bounded. While in the definition of $Q_{\epsilon}[\varphi]$, the singular terms $\nabla{v}_{1}$ and $\nabla{v}_{2}$ are used. In fact, there may exist a boundary data $\varphi$ such that $\mathcal{B}_{0}[\varphi]=0$, but it is easy to find another $\varphi$ such that $\mathcal{B}_{0}[\varphi]\neq0$ by a perturbation argument. 
\end{remark}

\begin{remark}
We would like to point out that from \eqref{nablau_bar} our asymptotic formula \eqref{equ_thm1} and \eqref{equ2_thm1} are actually pointwise expressions near the origin. This is different with the results in \cite{gn,g}, where the norm $\|\nabla{u}\|_{L^{\infty}(\Omega_{\delta})}$ is considered.
\end{remark}

From \eqref{kappa}, one can see the constant $\kappa_{n}$ depends on the curvature of $\partial{D}_{1}^{0}$ and $\partial{D}_{2}^{0}$ at the origin. For example, if the mean curvature $\lambda_{1}\lambda_{2}\rightarrow\,0$, then the quatity $\frac{1}{\kappa_{3}}$ in \eqref{equ2_thm1} tends to zero as well. While, when $\partial{D}_{1}^{0}$ and $\partial{D}_{2}^{0}$ are relatively convex of order $m>2$, especially when there exist a constant $\lambda>0$ such that
\begin{equation}\label{h1h24}
(h_{1}-h_{2})(x')=\lambda|x'|^{m},\quad\,m>2,~\mbox{for}~|x'|<R_{0},
\end{equation}
that is, their relative curvature vanishes. This will cause the blow-up rate to change. In order to reveal the relation between the convexity and the blow-up rate for particles with zero curvature at the point of the closest distance, we here restrict our consideration only to this symmetric case \eqref{h1h24}. For more generalized $m$-convex inclusions cases, the same assertions should also be true. For simplicity, we also assume that
\begin{equation}\label{h1h25}
|\nabla_{x'}h_{1}|,|\nabla_{x'}h_{2}|\leq\,C|x'|^{m-1},~\mbox{for}~|x'|<R_{0}.
\end{equation}

We define
$$\rho_{n}^{m}(\varepsilon)=
\begin{cases}
\varepsilon^{1-\frac{n-1}{m}}&\mbox{for}~m>n-1,~n\geq3\\
|\log\varepsilon|^{-1}&\mbox{for}~m=n-1,~n\geq3\\
\varepsilon^{1-\frac{1}{m}}&\mbox{for}~m\geq2,~n=2.
\end{cases}
$$

\begin{theorem}\label{thmm}
Let $D$, $D_{1}$, $D_{2}$ be of $C^{2,\alpha}$ and satisfy \eqref{h1h24} and \eqref{h1h25} with $m\geq2$ if $n=2$, $m\geq\,n-1$ if $n\geq3$, $\varphi\in{C}^{2}(\partial{D})$. Assume that
$u\in{H}^{1}(D)\cap{C}^{1}(\overline{\Omega})$
is the solution to \eqref{equinfty2}. Then for $\varphi$ such that $\mathcal{B}_{0}[\varphi]\neq0$, we have

(i) if $m\geq\,2(n-1)$, $n\geq2$, 
\begin{equation}\label{equ_thmm}
\nabla{u}=\frac{\mathcal{B}_{0}[\varphi]\rho_{n}^{m}(\varepsilon)}{\mathcal{L}\,\lambda^{\frac{n-1}{m}}}
\nabla\bar{v}_{1}+O(1)\|\varphi\|_{C^{2}(\partial{D})},\quad\,\mbox{in}~~\Omega_{R_{0}};
\end{equation}

(ii) if $n-1\leq\,m<2(n-1)$ and $n\geq3$,
\begin{equation}\label{equ2_thmm}
\nabla{u}=\mathcal{B}_{0}[\varphi]\frac{\rho_{n}^{m}(\varepsilon)}{\mathcal{L}\,\lambda^{\frac{n-1}{m}}}\Big(1+O\big(\rho_{n}^{m}(\varepsilon)\big)\Big)
\nabla\bar{v}_{1}+O(1)\|\varphi\|_{C^{2}(\partial{D})},\quad\,\mbox{in}~~\Omega_{R_{0}},
\end{equation}
where $\mathcal{L}$ is a constant depending only on $m$ and $n$.
\end{theorem}

\begin{remark}
In some sense Theorem \ref{thmm} could be regard as an extension of an 2D asymptotic formula (21) in \cite{kleey}, 
\begin{equation}\label{kleey1}
\nabla u=\alpha_{0}\nabla q_{\varepsilon}+O(1),
\end{equation}
where $q_{\varepsilon}$ is a singular function in dimension two, with $\nabla q_{\varepsilon}\sim\frac{1}{\sqrt{\varepsilon}}$. The conclusions in Theorem \ref{thmm} hold in dimensions two and three. Moreover, they show that the blow-up rate of $|\nabla{u}|$ at the origin is $\frac{\rho_{n}^{m}(\varepsilon)}{\varepsilon}$, which depends on the space dimension and the order of the convexity of the inclusions. Especially, in $\mathbb{R}^{n}$, when the convexity of inclusions is different, the blow-up rate is different. In this sense, when we use a ball (with $2$-convexity) to approximate an arbitrary convex inclusion, the error in general will be large, unless its convexity is also of order $2$. 
\end{remark}

\subsection{The outline of the proof of Theorems \ref{thm1} and \ref{thmm}}%\label{sec_thm1}

In this section we list the strategy and main ingredients of the proof of Theorem \ref{thm1} and \ref{thmm}.  Without loss of generality, we assume that $\|\varphi\|_{C^{2}(\partial{D})}=1$, by considering $u/\|\varphi\|_{C^{2}(\partial{D})}$ if $\|\varphi\|_{C^{2}(\partial{D})}>0$. If $\varphi~\big|_{\partial{D}}=0$ then $u\equiv0$.

Using the trace embedding theorem and $\|u\|_{H^{1}(\Omega)}\leq\,C$ (independent of $\varepsilon$),  we have
\begin{equation}\label{C1C2}
|C_{1}|+|C_{2}|\leq\,C.
\end{equation}
In view of \eqref{decomposition_u} and the third line of \eqref{equinfty2}, the constants $C_{1}-C_{2}$ is determined by the following linear system
\begin{equation}\label{sysC1C2*}
(C_{1}-C_{2})\int_{\partial{D}_{i}}\frac{\partial{v}_{1}}{\partial\nu^{-}}+\int_{\partial{D}_{i}}\frac{\partial{v}_{b}}{\partial\nu^{-}}=0,\quad i=1,2.
\end{equation}
If $\int_{\partial{D}_{1}}\frac{\partial v_{1}}{\partial\nu^{-}}\neq 0$, then from \eqref{sysC1C2*},  
\begin{equation}\label{C1-C2}
C_{1}-C_{2}=\dfrac{-\int_{\partial{D}_{1}}\frac{\partial v_{b}}{\partial\nu^{-}}}{\int_{\partial{D}_{1}}\frac{\partial v_{1}}{\partial\nu^{-}}}.
\end{equation}

In the following we estimate the two terms $\int_{\partial{D}_{1}}\frac{\partial{v}_{1}}{\partial\nu}$ and $\int_{\partial{D}_{1}}\frac{\partial{v}_{b}}{\partial\nu^{-}}$, respectively. First,
by the definition of $v_{1}$ and integration by parts, we have
$$\int_{\partial{D}_{1}}\frac{\partial v_{1}}{\partial\nu^{-}}=\int_{\partial\Omega}v_{1}\frac{\partial v_{1}}{\partial\nu}=\int_{\Omega}|\nabla{v}_{1}|^{2}.$$

\begin{thmA}(\cite{lly})\label{thm_energy}
For $n=2,3$, assume $D_{1}$, $D_{2}$ are of $C^{k,1}$, $k\geq3$ and satisfy \eqref{h1h23}. Then there exists a constant $M$, depending only on $D_{1}^{0},D_{2}^{0}$ and $\Omega$, such that
\begin{equation}\label{asym1}
\int_{\Omega}|\nabla{v}_{1}|^{2}
-\Big(\frac{\kappa_{n}}{\rho_{n}(\varepsilon)}+M\Big)=O\Big(E_{n}(\varepsilon)\Big),
\end{equation}
where
$$
E_{n}(\varepsilon)=\begin{cases}
\varepsilon^{\frac{1}{4}-\frac{1}{2k}}&n=2,\\
\varepsilon^{\frac{k-1}{2k}}|\log\varepsilon|&n=3.
\end{cases}
$$
\end{thmA}

For $m$-convexity inclusions with zero-curvature, in order to extend Theorem A to all dimensions, we need the following proposition, which shows that $\nabla\bar{v}_{1}$ is the main singular part of $\nabla{v}_{1}$ in $\Omega_{R_{0}}$.

\begin{prop}\label{prop_v1}
For $n\geq2$, assume $D_{1}$, $D_{2}$ are of $C^{2,\alpha}$ and satisfy \begin{equation}\label{h1h2m}
\frac{1}{C}|x'|^{m}\leq(h_{1}-h_{2})(x')\leq\,C|x'|^{m},\quad\,m\geq\max\{2,n-1\},
\end{equation}
and \eqref{h1h25}. Let $v_{1}\in{H}^1(D)$ be the
weak solution of \eqref{equv1}. Then
\begin{equation}\label{nabla_w_in}
\|\nabla(v_{1}-\bar{v}_{1})\|_{L^{\infty}(\Omega)}\leq\,C.
\end{equation}
\end{prop}

\begin{theorem}\label{thm_energy2}
For $n\geq2$, assume $D_{1}$, $D_{2}$ are of $C^{2,\alpha}$, and satisfy \eqref{h1h2m} and \eqref{h1h24}-\eqref{h1h25}. Then there exists a constant $M$, depending only on $D_{1}^{0},D_{2}^{0}$ and $\Omega$, such that
\begin{equation}\label{asym2}
\int_{\Omega}|\nabla{v}_{1}|^{2}
=\frac{\mathcal{L}\,\lambda^{\frac{n-1}{m}}}{\rho^{m}_{n}(\varepsilon)}+M+O\Big(E_{n}^{m}(\varepsilon)\Big),
\end{equation}
where $\mathcal{L}$ is a constant depending only on $m$ and $n$, and
\begin{equation}\label{Enm}
E_{n}^{m}(\varepsilon)=
\begin{cases}\varepsilon^{\frac{1}{4m}}~&\mbox{if}~m\geq2,~n=2,\\
\max\{\varepsilon^{\frac{1}{n-1}},\varepsilon^{\frac{1}{4}}|\log\varepsilon|\}~&\mbox{if}~m=\,n-1,~n\geq3,\\
\varepsilon^{\frac{n-1}{4m}}~&\mbox{if}~m>\,n-1,~n\geq3.
\end{cases} 
\end{equation}
\end{theorem}
From \eqref{asym2}, one can see that the energy aggregation of $v_{1}$ depends on the local geometry of the inclusions, such as $\lambda$, and the order of convexity $m$. The proof of Theorem \ref{thm_energy2} will be given in Section \ref{sec_asym}.
 
On the other hand, since $\Delta{v}_{b}=0$ in $D$ with $v_{b}=C_{2}$ on $\partial{D}_{1}\cup\partial{D}_{2}$, it follows from the standard elliptic theory that

\begin{theorem}\label{thm2}
Suppose that $0<\varepsilon<1/2$ sufficiently small. There are two positive constants $A,C$, independent of $\varepsilon$, such that
\begin{equation}\label{thm2.5_equ}
|\nabla v_{b}(x',x_{n})|\leq\,C\exp\left(-\frac{A}{(\varepsilon+|x'|^{m})^{1-1/m}}\right)\|v_{b}\|_{L^{2}(\Omega)},\quad\forall~(x',x_{n})\in\,\Omega_{R_{0}}.
\end{equation}
\end{theorem}
Theorem \ref{thm2} implies that
\begin{equation}\label{v2_bdd}
\|\nabla{v}_{b}\|_{L^{\infty}(\Omega_{R_{0}})}\leq\,C.
\end{equation}
So that, combining with the classical elliptic theory,
\begin{equation}\label{v2_bdd2}
\|\nabla{v}_{b}\|_{L^{\infty}(\Omega)}\leq\,C.
\end{equation}

Denote 
\begin{equation}\label{Qphi1}
\mathcal{B}_{\varepsilon}[\varphi]:=-\int_{\partial{D}_{1}}\frac{\partial v_{b}}{\partial\nu^{-}}.
\end{equation}
Substituting \eqref{asym2} and \eqref{v2_bdd2} into \eqref{C1-C2}, we have
\begin{equation}\label{C1-C2_2}
|C_{1}-C_{2}|\leq\,C\rho_{n}^{m}(\varepsilon).
\end{equation}
By using \eqref{C1-C2_2}, we prove, see Lemma \ref{lem6.2} below, that
\begin{equation}\label{C1-C2_3}
\Big|C_{i}-C_{0}\Big|\leq\,C\rho^{m}_{n}(\varepsilon),\quad\,i=1,2.
\end{equation}
This shows that $u_{0}$ defined by \eqref{equ_u0} is the limit of $u$ defined by \eqref{equinfty2}. Furthermore, as for the convergent rate of $\mathcal{B}_{\varepsilon}[\varphi]$ to $\mathcal{B}_{0}[\varphi]$, we have the following estimate.

\begin{prop}\label{prop2}
Let $\mathcal{B}_{\varepsilon}[\varphi]$ and $\mathcal{B}_{0}[\varphi]$ be defined by \eqref{Qphi1} and \eqref{Qphi0}, respectively. Then 

(i) under the assumptions of Theorem \ref{thm1}, we have
\begin{align}\label{B_converge0}
\mathcal{B}_{\varepsilon}[\varphi]-\mathcal{B}_{0}[\varphi]=
O\Big(\rho^{2}_{n}(\varepsilon)\Big),\quad\,n=2,3;
\end{align}

(ii) under the assumptions of Theorem \ref{thmm}, we have
\begin{align}\label{B_converge}
\mathcal{B}_{\varepsilon}[\varphi]-\mathcal{B}_{0}[\varphi]=
O\Big(\rho^{m}_{n}(\varepsilon)\Big)\quad\,m\geq\max\{2,n-1\},~n\geq2.
\end{align}
\end{prop}

This convergence rate is optimal because of \eqref{C1-C2_3}. The proof of Proposition \ref{prop2} will be given in Section \ref{Sec_prop2}. We are now in position to prove Theorems \ref{thm1} and \ref{thmm}.

\begin{proof}[Proof of Theorem \ref{thm1}] By using \eqref{decom_u2}, \eqref{nabla_w_in} and \eqref{v2_bdd},
\begin{equation*}\label{decomposition_u3}
\nabla{u}=(C_{1}-C_{2})\nabla\bar{v}_{1}+O(1),\quad\,\mbox{in}~~\Omega_{R_{0}}.
\end{equation*}
It follows from Proposition \ref{prop2} that
\begin{align*}
C_{1}-C_{2}=\dfrac{-\int_{\partial{D}_{1}}\frac{\partial v_{b}}{\partial\nu^{-}}}{\int_{\partial{D}_{1}}\frac{\partial v_{1}}{\partial\nu^{-}}}=\frac{\mathcal{B}_{\varepsilon}[\varphi]}{\int_{\Omega}|\nabla{v}_{1}|^{2}}.
\end{align*}
Thus, using \eqref{asym1} and \eqref{B_converge}, 
\begin{align}\label{equ-final}
\nabla{u}(x)=&(C_{1}-C_{2})\nabla\bar{v}_{1}(x)+O(1)\nonumber\\
=&\frac{\mathcal{B}_{0}[\varphi]+O\big(\rho_{n}(\varepsilon)\big)}{\frac{\kappa_{n}}{\rho_{n}(\varepsilon)}+M+O\big(E_{n}(\varepsilon)\big)}
\nabla\bar{v}_{1}(x)+O(1).
\end{align}
In view of the definition of $\rho_{n}(\varepsilon)$, \eqref{def_rhon}, Theorem \ref{thm1} follows easily from the above.
\end{proof}

\begin{proof}[Proof of Theorem \ref{thmm}] Replacing \eqref{asym1} by \eqref{asym2} in \eqref{equ-final}, we have
\begin{align*}
\nabla{u}(x)=&(C_{1}-C_{2})\nabla\bar{v}_{1}(x)+O(1)\nonumber\\
=&\frac{\mathcal{B}_{0}[\varphi]+O\big(\rho^{m}_{n}(\varepsilon)\big)}{\frac{\mathcal{L}\,\lambda^{\frac{n-1}{m}}}{\rho_{n}^{m}(\varepsilon)}+M+O\big(E_{n}^{m}(\varepsilon)\big)}
\nabla\bar{v}_{1}(x)+O(1).
\end{align*}
The proof is completed by a direct computation.
\end{proof}

The rest of this paper is organized as follows. We establish the pointwise upper and lower bound estimates of $|\nabla{v}_{1}|$ in Section \ref{sec_v1}. The asymptotics of the energy of $v_{1}$ for $m$-convex inclusions is proved in Section \ref{sec_asym}. The proof of Theorem \ref{thm2} and Proposition \ref{prop2} is given in Section \ref{Sec_prop2}.

\bigskip

\section{The gradient estimates of ${v}_{1}$}\label{sec_v1}

This section is devoted to the estimates of $|\nabla{v}_{1}|$ for $m$-convexity inclusions with zero-curvature.

\begin{proof}[Proof of Proposition \ref{prop_v1}]

For simplicity, denote
\begin{equation*}\label{def_w}
w:=v_{1}-\bar{v}_{1}.
\end{equation*}
By the definition of $v_{1}$ in \eqref{equv1}, and $v_{1}=\bar{v}_{1}$ on $\partial{D}_{1}\cup\partial{D}_{2}\cup\partial{D}$, we have
\begin{equation}\label{w20}
\begin{cases}
-\Delta{w}=\Delta\bar{v}_{1}&\mbox{in}~\Omega,\\
w=0&\mbox{on}~\partial \Omega.
\end{cases}
\end{equation}
In view of \eqref{ubar} and \eqref{nablau_bar_outside}, 
\begin{equation}\label{nabla_u_out}
\|\bar{v}_{1}\|_{C^{2,\alpha}(\Omega\setminus\Omega_{R_{0}/3})}\leq\,C.\end{equation}
Using the standard elliptic theory, we have
\begin{equation}\label{nabla_w_out}
|w|+\left|\nabla{w}\right|+\left|\nabla^{2}w\right|\leq\,C,
\quad\mbox{in}~~\Omega\setminus\Omega_{R_{0}/2}.
\end{equation}
Thus, to show \eqref{nabla_w_in}, we only need to prove
$$\left\|\nabla{w}\right\|_{L^{\infty}(\Omega_{R_{0}/2})}\leq\,C.$$

First, we claim that 
\begin{equation}\label{energy_w}
\int_{\Omega}\left|\nabla{w}\right|^{2}\leq\,C.
\end{equation}
Indeed, by the maximum principle, we have $0<v_{1}<1$. Becuase $\bar{v}_{1}$ is also bounded, 
\begin{equation}\label{w_bdd}
\|w\|_{L^{\infty}(\Omega)}\leq\,C.
\end{equation}
A direct computation yields, 
\begin{equation}\label{delta_ubar}
|\Delta\bar{v}_{1}(x)|\leq\frac{C}{\delta(x')},\quad\mbox{where}~\delta(x')=\varepsilon+h_{1}(x')-h_{2}(x'),\qquad\,x\in\Omega_{R_{0}}.
\end{equation}
Now multiplying the equation in \eqref{w20} by $w$, integrating by parts, and making use of \eqref{nabla_u_out}, \eqref{w_bdd} and \eqref{delta_ubar}, 
\begin{align*}\label{energy_w_1}
\int_{\Omega}|\nabla{w}|^{2}
=\int_{\Omega}w\left(\Delta\bar{v}_{1}\right)
\leq\|w\|_{L^{\infty}(\Omega)}\left(\int_{\Omega_{R_{0}}}|\Delta\bar{v}_{1}|+C\right)\leq\,C.
\end{align*}
Thus, \eqref{energy_w} is proved.

 For $0<t<s<R_{0}$, let $\eta$ be a smooth cutoff function satisfying $\eta(x')=1$ if $|x'-z'|<t$, $\eta(x')=0$ if $|x'-z'|>s$, $0\leq\eta(x')\leq1$ if $t\leq|x'-z'|\leq\,s$, and $|\nabla_{x'}\eta(x')|\leq\frac{2}{s-t}$. Multiplying the equation in \eqref{w20} by $w\eta^{2}$ and integrating by parts leads  to the following Caccioppolli's type inequality
\begin{equation}\label{FsFt11}
\int_{\Omega_{t}(z')}|\nabla{w}|^{2}\leq\,\frac{C}{(s-t)^{2}}\int_{\Omega_{s}(z')}|w|^{2}
+C(s-t)^{2}\int_{\Omega_{s}(z')}\left|\Delta\bar{v}_{1}\right|^{2},
\end{equation}
where
$$ \Omega_r(z'):=\left\{(x',x_{n})\in \mathbb{R}^{n}~\big|~h_{2}(x')<x_{n}<\varepsilon+h_{1}(x'),~|x'-z'|<r\right\}.$$

The rest of the proof is divided into two steps. By an iteration technique developed in \cite{bll}, we first have

\noindent{\bf STEP 1.} Proof of
\begin{equation}\label{energy_w_inomega_z1}
\int_{\Omega_{\delta}(z')}\left|\nabla{w}\right|^{2}dx\leq\,C\delta(z')^{n},\quad\mbox{for}~~n\geq2,
\end{equation}
where
\begin{equation*}\label{delta_x}
\delta=\delta(z'):=\varepsilon+h_{1}(z')-h_{2}(z'),\quad\mbox{for}~(z',z_{n})\in\Omega_{R_{0}}.
\end{equation*}

We adapt the iteration technique developed in \cite{bll} and give a unified iteration process. For $0<s<|z'|\leq\,R_{0}/2$, we note that by using H\"{o}lder inequality,
\begin{align*}%\label{energy_w_square}
\int_{\Omega_{s}(z')}|w|^{2}=\int_{\Omega_{s}(z')}\Big(\int_{h_{2}(x')}^{x_{n}}\partial_{x_{n}}w\Big)^{2}
\leq&\,C\delta(z')^{2}\int_{\Omega_{s}(z')}|\nabla{w}|^{2}, \quad\mbox{if}~\,0<s<\frac{2|z'|}{3}.
\end{align*}
Substituting it into \eqref{FsFt11} and denoting
$$F(t):=\int_{\Omega_{t}(z')}|\nabla{w}|^{2},$$
we have
\begin{equation}\label{tildeF111}
F(t)\leq\,\left(\frac{C_{0}\delta(z')}{s-t}\right)^{2}F(s)+C(s-t)^{2}\int_{\Omega_{s}(z')}\left|\Delta\bar{v}_{1}\right|^{2},
\quad\forall~0<t<s<\frac{2|z'|}{3},
\end{equation}
where $C_0$ is a fixed positive universal constant.

Let $k=\left[\frac{\max\{\varepsilon^{1/m},|z'|\}}{4C_{0}\delta(z')}\right]$ and $t_{i}=\delta+2C_{0}i\,\delta(z')$, $i=0,1,2,\cdots,k$. Taking $s=t_{i+1}$ and $t=t_{i}$ in \eqref{tildeF111}, and in view of \eqref{delta_ubar}, 
\begin{equation}\label{integal_Lubar11}
\int_{\Omega_{t_{i+1}}(z')}\left|\Delta\bar{v}_{1}\right|^{2}\leq\int_{|x'-z'|<t_{i+1}}\frac{C}{\delta(x')}dx'\leq\frac{Ct_{i+1}^{n-1}}{\delta(z')}\leq\,C(i+1)^{n-1}\delta(z')^{(n-2)}.
\end{equation}
We obtain an iteration formula
$$F(t_{i})\leq\,\frac{1}{4}F(t_{i+1})+C(i+1)^{n-1}\delta(z')^{n}.$$
After $k$ iterations, using (\ref{energy_w}),
\begin{eqnarray*}
F(t_{0}) \leq (\frac{1}{4})^{k}F(t_{k})+C\delta(z')^{n}\sum_{l=1}^{k}(\frac{1}{4})^{l-1}l^{n-1}
\leq C\delta(z')^{n}.
\end{eqnarray*}
This implies that \eqref{energy_w_inomega_z1}.

\noindent{\bf STEP 2.} Next, we use Sobolev embedding theorem and classical $L^{p}$ estimates for elliptic equations to prove \eqref{nabla_w_in}.

By using the following scaling and translating of variables
\begin{equation*}%\label{changeofvariant}
 \left\{
  \begin{aligned}
  &x'-z'=\delta(z') y',\\
  &x_n=\delta(z') y_n,
  \end{aligned}
 \right.
\end{equation*}
then $\Omega_{\delta(z')}(z')$ becomes $Q_{1}$, where for $r\leq1$,
$$Q_{r}=\left\{y\in\mathbb{R}^{n}\Big|~\frac{1}{\delta(z')}h_{2}(\delta(z'){y}'+z')<y_{n}
<\frac{\varepsilon}{\delta(z')}+\frac{1}{\delta(z')}h_{1}(\delta(z'){y}'+z'),|y'|<r\right\},$$ and the top and
bottom boundaries respectively
become
$$
y_n=\hat{h}_{1}(y'):=\frac{1}{\delta(z')}
\left(\varepsilon+h_{1}(\delta(z')\,y'+z')\right),\quad|y'|<1,$$
and
$$y_n=\hat{h}_{2}(y'):=\frac{1}{\delta(z')}h_{2}(\delta(z')\,y'+z'), \quad |y'|<1.
$$
 Then
$$\hat{h}_{1}(0')-\hat{h}_{2}(0'):=\frac{1}{\delta(z')}\left(\varepsilon+h_{1}(z')-h_{2}(z')\right)=1,$$
and by \eqref{h1h20},
$$|\nabla_{x'}\hat{h}_{1}(0')|+|\nabla_{x'}\hat{h}_{2}(0')|\leq\,C|z'|^{m-1},
\quad|\nabla_{x'}^{2}\hat{h}_{1}(0')|+|\nabla_{x'}^{2}\hat{h}_{2}(0')|\leq\,C|z'|^{m-2}.$$
Since $R_{0}$ is small, $\|\hat{h}_{1}\|_{C^{1,1}((-1,1)^{n-1})}$ and $\|\hat{h}_{2}\|_{C^{1,1}((-1,1)^{n-1})}$ are small and $Q_{1}$ is essentially a unit square (or a unit cylinder for $n=3$) as far as applications of the Sobolev embedding theorem and classical $L^{p}$ estimates for elliptic equations are concerned.
Let
\begin{equation*}%\label{def_U}
\overline{V}_{1}(y', y_n):=\bar{v}_{1}(z'+\delta(z'){y}',\delta(z'){y}_{n}),\quad\,W(y', y_n):=w(z'+\delta(z'){y}',\delta(z'){y}_{n}),
\quad\,y\in{Q}'_{1},
\end{equation*}
then by \eqref{w20},
\begin{align*}%\label{syswidew2}
-\Delta{W}
=\Delta\overline{V}_{1},
\quad\quad\,y\in{Q_{1}},
\end{align*}
where
$$\left|\Delta\overline{V}_{1}\right|=\delta(z')^{2}\left|\Delta\bar{v}_{1}\right|.$$

Since $W=0$ on the top and
bottom boundaries of $Q_{1}$, using
the Poincar\'{e} inequality,
$$\left\|W\right\|_{H^{1}(Q_{1})}\leq\,C\left\|\nabla{W}\right\|_{L^{2}(Q_{1})}.$$
By $W^{2,p}$ estimates for elliptic equations (see e.g. \cite{gt}), 
the Sobolev embedding theorems, and using the bootstrap argument, with $p>n$,
\begin{align*}
\left\|\nabla{W}\right\|_{L^{\infty}(Q_{1/2})}\leq\,
C\left\|W\right\|_{W^{2,p}(Q_{1/2})}
\leq\,C\left(\left\|\nabla{W}\right\|_{L^{2}(Q_{1})}+\left\|\Delta\overline{V}_{1}\right\|_{L^{\infty}(Q_{1})}\right).
\end{align*}
It follows from $\nabla{W}=\delta\nabla{w}$ and \eqref{delta_ubar},\eqref{energy_w_inomega_z1} that
\begin{align}\label{equ3.18}
&\left\|\nabla{w}\right\|_{L^{\infty}(\Omega_{\delta(z')/2}(z'))}\nonumber\\
&\leq\,
C\left(\delta(z')^{-n/2}\left\|\nabla{w}\right\|_{L^{2}(\Omega_{\delta(z')}(z'))}
+\delta(z')\left\|\Delta\bar{v}_{1}\right\|_{L^{\infty}(\Omega_{\delta(z')}(z'))}\right)\leq\,C.
\end{align}
Proposition \ref{prop_v1} is established.
\end{proof}

\begin{remark}%\label{rem31}
We point out that the estimate involving $\Delta\bar{v}_{1}$ is very crucial in the above proof, such as \eqref{integal_Lubar11} and \eqref{equ3.18}, for $\int_{\Omega_{t_{i+1}}(z')}\left|\Delta\bar{v}_{1}\right|^{2}
$ and $\delta(z')\left\|\Delta\bar{v}_{1}\right\|_{L^{\infty}(\Omega_{\delta(z')}(z'))}$. 
\end{remark}

An immediate consequence of Proposition \ref{prop_v1} is that
\begin{corollary} Under the assumption as in Proposition \ref{prop_v1},
\begin{equation}\label{nabla_v1_in}
\frac{1}{C(\varepsilon+(h_{1}-h_{2})(x'))}\leq|\nabla{v}_{1}(x',x_{n})|\leq\frac{C}{\varepsilon+ (h_{1}-h_{2})(x')},\qquad~(x',x_{n})\in\Omega_{R_{0}},
\end{equation}
and
\begin{equation}\label{nabla_v1_o}
\|\nabla v_{1}\|_{L^{\infty}(\Omega\setminus\Omega_{R_{0}})}\leq\,C.
\end{equation}
\end{corollary}

\bigskip

\section{Proof of Theorem \ref{thm_energy2}}\label{sec_asym}

Define $v_{1}^{0}$ to be the solution of the limiting problem
\begin{equation}\label{equv10}
\begin{cases}
\Delta{v}_{1}^{0}=0&\mbox{in}~\Omega^{0},\\
v_{1}^{0}=1&\mbox{on}~\partial{D}_{1}^{0}\setminus\{0\},\\
v_{1}^{0}=0&\mbox{on}~\partial{D}_{2}^{0}\cup\partial{D}.
\end{cases}
\end{equation}
Similarly as $\bar{v}_{1}$, we construct an auxiliary function $\bar{v}_{1}^{0}$, such that $\bar{v}_{1}^{0}=1$ on $\partial{D}_{1}^{0}\setminus\{0\}$, $\bar{v}_{1}^{0}=0$ on $\partial{D}_{2}^{0}\cup\partial{D}$,
\begin{equation}\label{def_barv10}
\bar{v}_{1}^{0}=\frac{x_{n}-h_{2}(x')}{(h_{1}-h_{2})(x')}\quad\mbox{in}~\Omega^{0}_{R_{0}}:=\Big\{(x',x_{n})\big|~h_{2}(x')\leq\,x_{n}\leq\,h_{1}(x'),~|x'|\leq\,R_{0}\Big\},
\end{equation}
and $\|\bar{v}_{1}^{0}\|_{C^{2,\alpha}(\Omega^{0}\setminus\Omega^{0}_{R_{0}})}\leq\,C$. It is easy to see that
\begin{equation}\label{nablau_starbar}
\left|\partial_{x'}\bar{v}_{1}^{0}(x)\right|\leq\frac{C}{|x'|},\qquad
\partial_{x_{n}}\bar{v}_{1}^{0}(x)=\frac{1}{(h_{1}-h_{2})(x')},\quad~x\in\Omega_{R_{0}}^{0}\setminus\{0\}.
\end{equation}

It follows from the proof of Proposition \ref{prop_v1} that 
\begin{equation}\label{nablaw_0star}
\left\|\nabla(v_{1}^{0}-\bar{v}_{1}^{0})\right\|_{L^{\infty}(\Omega^{0})}\leq\,C.
\end{equation}
This shows that $\nabla\bar{v}_{1}^{0}$ is also the main term of $\nabla{v}_{1}^{0}$.

\begin{lemma}\label{lem_29}
Let $v_{1}$ and $v_{1}^{0}$ be defined by \eqref{equv1} and \eqref{equv10}, respectively. Then
\begin{equation}\label{v1-v1*}
\|v_{1}-v_{1}^{0}\|_{L^{\infty}\Big(\Omega\setminus{\big(D_{1}\cup{D}_{2}\cup{D}_{1}^{0}\cup\Omega_{\varepsilon^{1/(2m)}}\big)}\Big)}\leq\,C\varepsilon^{1/2},\qquad\,i=1,2.
\end{equation}
\end{lemma}

\begin{proof}
We will first consider the difference $v_{1}-v_{1}^{0}$ on the boundary of $\Omega\setminus(D_{1}\cup{D}_{2}\cup{D}_{1}^{0}\cup\Omega_{\varepsilon^{1/(2m)}})$, then use the maximum principle to obtain \eqref{v1-v1*}.

{\bf STEP 1.}
Obviously, 
\begin{equation}\label{lem3.1_equ1}
v_{1}-v_{1}^{0}=0,\quad\mbox{ on}~ \partial{D}_{2}\cup\partial{D}.\end{equation}
In the following we only need to deal with the boundary $\partial(D_{1}\cup{D}_{1}^{0})$. We divide it into two parts: (a) $\partial{D}_{1}^{0}\setminus{D}_{1}$ and (b) $\partial{D}_{1}\setminus{D}_{1}^{0}$.

(a) When $x\in\partial{D}_{1}^{0}\setminus{D}_{1}$, we introduce a cylinder
$$\mathcal{C}_{r}:=\left\{x\in\mathbb{R}^{n}~\big|~2\min_{|x'|=r}h_{2}(x')\leq\,x_{n}\leq\varepsilon+2\max_{|x'|=r}h_{1}(x'),~|x'|<r\right\},\quad\,r\leq\,R_{0}.$$
(a1) For $x\in\partial{D}_{1}^{0}\cap(\mathcal{C}_{R_{0}}\setminus\mathcal{C}_{\varepsilon^{1/(2m)}})$, using $v_{1}^{0}=1$ on $\partial{D}_{1}^{0}$ and $v_{1}=1$ on $\partial{D}_{1}$, by mean value theorem and estimate \eqref{nabla_v1_in}, we have, for some $\theta_{\varepsilon}\in(0,1)$
\begin{align*}
|v_{1}(x)-v_{1}^{0}(x)|=|v_{1}(x)-1|=&\left|v_{1}(x',h_{1}(x'))-v_{1}(x',\varepsilon+h_{1}(x'))\right|\\
=&\left|\partial_{x_{n}}v_{1}(x',\theta_{\varepsilon}\varepsilon+h_{1}(x'))\right|\cdot\varepsilon\\
\leq&\,\frac{C\varepsilon}{\varepsilon+|x'|^{m}}
\leq\,C\varepsilon^{1/2}.
\end{align*}
(a2) For $x\in\partial{D}_{1}^{0}\cap(\Omega\setminus\Omega_{R_{0}})$, there exists $y_{\varepsilon}\in\partial{D}_{1}\cap\overline{\Omega\setminus\Omega_{R_{0}/2}}$ such that $|x-y_{\varepsilon}|<C\varepsilon$ (note that $v_{1}(y_{\varepsilon})=1$). By \eqref{nabla_v1_o} and mean value theorem again, for some $\theta_{\varepsilon}\in(0,1)$
$$|v_{1}(x)-v_{1}^{0}(x)|=|v_{1}(x)-1|=|v_{1}(x)-v_{1}(y_{\varepsilon})|
\leq|\nabla{v}_{1}((1-\theta_{\varepsilon})x+\theta_{\varepsilon}{y}_{\varepsilon})||x-y_{\varepsilon}|\leq\,C\varepsilon.$$

(b) When $x\in\partial{D}_{1}\setminus{D}_{1}^{0}$, since $0<v_{1}<1$ in $\Omega$ and $\Delta{v}_{1}=0$ in $\Omega$, it follows from the boundary estimates of harmonic function that there exists $y_{x}\in\Omega$, $|y_{x}-x|\leq\,C\varepsilon$ such that $v_{1}(y_{x})=v_{1}^{0}(x)$. Using \eqref{nabla_v1_o} again, 
$$|v_{1}(x)-v_{1}^{0}(x)|=|v_{1}(x)-v_{1}(y_{x})|\leq\,\|\nabla{v}_{1}\|_{L^{\infty}(\Omega\setminus\Omega_{R_{0}})}|x-y_{x}|\leq\,C\varepsilon.$$
Therefore,
\begin{equation}\label{onboundaryD1}
|v_{1}(x)-v_{1}^{0}(x)|\leq\,C\varepsilon^{1/2},\quad\mbox{for}~x\in\partial(D_{1}\cup{D}_{1}^{0})\setminus\mathcal{C}_{\varepsilon^{1/(2m)}}.
\end{equation}

{\bf STEP 2.} We consider the lateral boundary of $\Omega_{\varepsilon^{1/(2m)}}^{0}$, 
$$S_{1/(2m)}:=\Big\{(x',x_{n})~\big| ~h_{2}(x')\leq\,x_{n}\leq\,h_{1}(x'),~|x'|=\varepsilon^{1/(2m)}\Big\},$$ By using \eqref{nabla_w_in} and $(v_{1}-\bar{v}_{1})=0$ on $\partial{D}_{2}$, we have, for $x\in\,S_{1/(2m)}$, 
\begin{equation}\label{part1}
|(v_{1}-\bar{v}_{1})(x)|\leq\left\|\nabla(v_{1}-\bar{v}_{1})\right\|_{L^{\infty}(S_{1/(2m)})}\left|(h_{1}-h_{2})(x')\right|\leq\,C|x'|^{m}\leq\,C\varepsilon^{1/2}.
\end{equation}
Similarly, since $(v_{1}^{0}-\bar{v}_{1}^{0})=0$ on $\partial{D}_{2}$, it follows from \eqref{nablaw_0star} and mean value theorem that for $x\in\,S_{1/(2m)}$, 
\begin{equation}\label{part2}
|(v_{1}^{0}-\bar{v}_{1}^{0})(x)|\leq\left\|\nabla(v_{1}^{0}-\bar{v}_{1}^{0})\right\|_{L^{\infty}(S_{1/m-\beta})}\left|(h_{1}-h_{2})(x')\right|\leq\,C|x'|^{m}\leq\,C\varepsilon^{1/2}.
\end{equation}
Since $\bar{v}_{1}=\bar{v}_{1}^{0}\equiv0$ on $\partial{D}_{2}$, then for $x\in\,S_{1/(2m)}$, 
\begin{align}\label{part3}
|(\bar{v}_{1}-\bar{v}_{1}^{0})(x)|
&\leq\left\|\partial_{x_{n}}(\bar{v}_{1}-\bar{v}_{1}^{0})\right\|_{L^{\infty}(S_{1/(2m)})}\left|(h_{1}-h_{2})(x')\right|\nonumber\\
&\leq\,C\max_{|x'|=\varepsilon^{1/(2m)}}\left\{\frac{1}{(h_{1}-h_{2})(x')}-\frac{1}{\varepsilon+(h_{1}-h_{2})(x')}\right\}|x'|^{m}\nonumber\\
&\leq\frac{C\varepsilon}{|x'|^{m}(\varepsilon+|x'|^{m})}|x'|^{m}\leq\,C\varepsilon^{1/2}.
\end{align}
Thus, combining \eqref{part1}, \eqref{part2} with \eqref{part3}, we have, for $x\in\,S_{1/(2m)}$, 
\begin{equation}\label{onboundaryS}
|(v_{1}-v_{1}^{0})(x)|\leq|(v_{1}-\bar{v}_{1})(x)|+|(\bar{v}_{1}-\bar{v}_{1}^{0})(x)|+|(\bar{v}_{1}^{0}-v_{1}^{0})(x)|\leq\,C\varepsilon^{1/2}.
\end{equation}

Finally, by \eqref{lem3.1_equ1}, \eqref{onboundaryD1} and \eqref{onboundaryS}, and applying the maximum principle to $(v_{1}-v_{1}^{0})$ on $\Omega\setminus{\big(D_{1}\cup{D}_{2}\cup{D}_{1}^{0}\cup\Omega_{\varepsilon^{1/(2m)}}\big)}$, we obtain
\eqref{v1-v1*}.
\end{proof}

If $\partial{D}_{1}$ and $\partial{D}_{2}$ are assumed to be fo $C^{2,\alpha}$ and satisfy \eqref{h1h24} and \eqref{h1h25}, then we have an improvement of Lemma \ref{lem_29} by interpolation.

\begin{lemma}\label{lem_49}
Assume that $v_{1}$ and $v_{1}^{0}$ are solution of \eqref{equv1} and \eqref{equv10}, respectively. If $\partial{D}_{1}^{0}$ and $\partial{D}_{2}^{0}$ are of $C^{2,\alpha}$ and satisfy \eqref{h1h24}--\eqref{h1h25}, then 
\begin{equation}\label{nabla_v1v10}
\begin{split}
&|\nabla{v}_{1}(x)|\leq\,C|x'|^{-m},~\,x\in\Omega_{R_0}\setminus\Omega_{\varepsilon^{1/(2m)}},\\
&|\nabla{v}_{1}^{0}(x)|\leq\,C|x'|^{-m},~\,x\in\Omega^{0}_{R_0}\setminus\Omega^{0}_{\varepsilon^{1/(2m)}};
\end{split}
\end{equation}
and
\begin{equation}\label{nabla_v1-v10}
|\nabla(v_{1}-v_{1}^{0})(x)|\leq\,C\varepsilon^{1/4}|x'|^{-m},\quad\mbox{in}~~\Omega^{0}_{R_0}\setminus\Omega^{0}_{\varepsilon^{1/(2m)}}.
\end{equation}
\end{lemma}

\begin{proof}
For $\varepsilon^{1/(2m)}\leq|z'|\leq\,R_0$, we make use of the change of variable
\begin{equation*}%\label{changeofvariable2}
 \left\{
  \begin{aligned}
  &x'-z'=|z'|^{m} y',\\
  &x_n=|z'|^{m} y_n,
  \end{aligned}
 \right.
\end{equation*}
to rescale $\Omega_{|z'|+ |z'|^{m}}\setminus\Omega_{|z'|}$ into an approximate unit-size cube (or cylinder) $Q_{1}$, and $\Omega^{0}_{|z'|+|z'|^{m}}\setminus\Omega^{0}_{|z'|}$ into $Q_{1}^{0}$. Let
$$V_{1}(y)=v_{1}(z'+|z'|^{m}y',|z'|^{m}y_{n}),\quad\mbox{in}~~Q_{1},$$
and
$$V_{1}^{0}(y)=v_{1}^{0}(z'+|z'|^{m}y',|z'|^{m}y_{n}),\quad\mbox{in}~~Q_{1}^{0}.$$
Since $0<V_{1},V_{1}^{0}<1$, using the standard elliptic theory, we have
\begin{equation*}%\label{resacaled_v1}
|\nabla^{2}V_{1}|\leq\,C,\quad\mbox{in}~~Q_{1},\quad
\mbox{and}~~
|\nabla^{2}V_{1}^{0}|\leq\,C,\quad\mbox{in}~~Q_{1}^{0}.
\end{equation*}
Interpolating it with \eqref{v1-v1*} yields
$$|\nabla(V_{1}-V_{1}^{0})|\leq\,C\varepsilon^{\frac{1}{2}(1-\frac{1}{2})}\leq\,C\varepsilon^{1/4},\quad\mbox{in}~~Q_{1}^{0}.$$
Thus, rescaling it back to $v_{1}-v_{1}^{0}$, we have \eqref{nabla_v1-v10} holds. 

By the way, we have
$$|\nabla{v}_{1}(x)|\leq\,C|z'|^{-m},\quad\,x\in\Omega_{|z'|+|z'|^{m}}\setminus\Omega_{|z'|},$$
and
$$\quad|\nabla{v}_{1}^{0}(x)|\leq\,C|z'|^{-m},\quad\,x\in\Omega^{0}_{|z'|+|z'|^{m}}\setminus\Omega^{0}_{|z'|},
$$
so \eqref{nabla_v1v10} follows.
\end{proof}

\begin{proof}[Proof of Theorem \ref{thm_energy2}]
To prove \eqref{asym2}, we divid the integral into two parts:
\begin{equation}\label{energy_v1}
\int_{\Omega}|\nabla{v}_{1}|^{2}
=\int_{\Omega\setminus\Omega_{\varepsilon^{\gamma}}}|\nabla{v}_{1}|^{2}+\int_{\Omega_{\varepsilon^{\gamma}}}|\nabla{v}_{1}|^{2}
=:\mathrm{I}+\mathrm{II},
\end{equation}
where we take $\gamma=\frac{1}{4m}$, for convenience.

\noindent{\bf STEP 1.} We first prove
\begin{equation}\label{equ4.12}
\mathrm{I}=\int_{\Omega\setminus\Omega_{\varepsilon^{\gamma}}}|\nabla{v}_{1}|^{2}=\int_{\Omega^{0}\setminus\Omega_{\varepsilon^{\gamma}}^{0}}|\nabla{v}_{1}^{0}|^{2}+O\Big(E_{n}^{m}(\varepsilon)\Big),
\end{equation}
where $E_{n}^{m}(\varepsilon)$ is defined in \eqref{Enm}. We divide term $\mathrm{I}$ further as follows:
\begin{align*}
\mathrm{I}=&\int_{\Omega\setminus\Omega_{R_{0}}}|\nabla{v}_{1}|^{2}
+\int_{\Omega_{R_{0}}\setminus\Omega_{\varepsilon^{\gamma}}}|\nabla{v}_{1}|^{2}=:\mathrm{I}_{1}+\mathrm{I}_{2}.
\end{align*}

First, for term $\mathrm{I}_{1}$, we claim that
\begin{equation}\label{claim1}
\mathrm{I}_{1}=M_{1}+O\left(\varepsilon^{1/4}\right),\qquad\,M_{1}:=\int_{\Omega^{0}\setminus\Omega^{0}_{R_0}}|\nabla{v}_{1}^{0}|^{2}.
\end{equation}
Indeed, since
$$\Delta(v_{1}-v_{1}^{0})=0,\quad\mbox{in}~~\Omega\setminus{\big(D_{1}\cup{D}_{1}^{0}\cup{D}_{2}\cup\Omega_{R_{0}/2}\big)},$$
and
$$0<v_{1},v_{1}^{0}<1,\quad\mbox{in}~~\Omega\setminus{\big(D_{1}\cup{D}_{1}^{0}\cup{D}_{2}\cup\Omega_{R_{0}/2}\big)},$$
it follows that provided $\partial{D}_{1}^{0}$, $\partial{D}_{2}^{0}$ and $\partial \Omega$ are of $C^{2,\alpha}$, 
$$|\nabla^{2}(v_{1}-v_{1}^{0})|\leq|\nabla^{2}v_{1}|+|\nabla^{2}v_{1}^{0}|\leq\,C,\quad\mbox{in}~~\Omega\setminus{\big(D_{1}\cup{D}_{1}^{0}\cup{D}_{2}\cup\Omega_{R_{0}}\big)},$$
where $C$ is independent of $\varepsilon$. By using an interpolation with \eqref{v1-v1*}, we have
\begin{equation*}\label{v1-v1*outside}
|\nabla(v_{1}-v_{1}^{0})|\leq\,C
\varepsilon^{1/2(1-\frac{1}{2})}\leq\,C
\varepsilon^{1/4},\quad\mbox{in}~~\Omega\setminus{\big(D_{1}\cup{D}_{1}^{0}\cup{D}_{2}\cup\Omega_{R_{0}}\big)}.
\end{equation*}
In view of the boundedness of $|\nabla{v}_{1}|$ in $D_{1}^{0}\setminus(D_{1}\cup\Omega_{R_{0}})$ and  $D_{1}\setminus{D_{1}^{0}}$, and $|D_{1}^{0}\setminus(D_{1}\cup\Omega_{R_{0}})|$ and $|D_{1}\setminus{D_{1}^{0}}|$ are less than $C\varepsilon$, 
\begin{align*}%\label{nablav1_part1}
&\mathrm{I}_{1}-M_{1}\\
=&\int_{\Omega\setminus{\big(D_{1}\cup{D}_{1}^{0}\cup{D}_{2}\cup\Omega_{R_{0}}\big)}}(|\nabla{v}_{1}|^{2}-|\nabla{v}_{1}^{0}|^{2})+\int_{D_{1}^{0}\setminus(D_{1}\cup\Omega_{R_{0}})}|\nabla{v}_{1}|^{2}+\int_{D_{1}\setminus{D_{1}^{0}}}|\nabla{v}_{1}^{0}|^{2}\nonumber\\
=&2\int_{\Omega\setminus{\big(D_{1}\cup{D}_{1}^{0}\cup{D}_{2}\cup\Omega_{R_{0}}\big)}}\nabla{v}_{1}^{0}\nabla(v_{1}-v_{1}^{0})+\int_{\Omega\setminus{\big(D_{1}\cup{D}_{1}^{0}\cup{D}_{2}\cup\Omega_{R_{0}}\big)}}|\nabla(v_{1}-v_{1}^{0})|^{2}+O(\varepsilon)\nonumber\\
=&O\left(\varepsilon^{1/4}\right).
\end{align*}
Thus, \eqref{claim1} is proved.

For $\mathrm{I}_{2}$, we will prove that
\begin{equation}\label{claim2}
\mathrm{I}_{2}=\mathrm{I}_{2}^{(0)}+E_{n,m}(\varepsilon),\qquad\,\mathrm{I}_{2}^{(0)}
:=\int_{\Omega^{0}_{R_0}\setminus\Omega^{0}_{\varepsilon^{\gamma}}}|\nabla{v}_{1}^{0}|^{2}.
\end{equation}
Indeed, 
\begin{align}\label{equ4.15}
\mathrm{I}_{2}-\mathrm{I}_{2}^{(0)}=&\int_{(\Omega_{R_{0}}\setminus\Omega_{\varepsilon^{\gamma}})\setminus(\Omega^{0}_{R_0}\setminus\Omega^{0}_{\varepsilon^{\gamma}})}|\nabla{v}_{1}|^{2}+\int_{\Omega^{0}_{R_0}\setminus\Omega^{0}_{\varepsilon^{\gamma}}}|\nabla(v_{1}-v^{0}_{1})|^{2}\nonumber\\
&+2\int_{\Omega^{0}_{R_0}\setminus\Omega^{0}_{\varepsilon^{\gamma}}}\nabla{v}_{1}^{0}\cdot\nabla(v_{1}-v^{0}_{1}).
\end{align}
For the first term in the right hand side of \eqref{equ4.15}, because the thickness of $(\Omega_{R_{0}}\setminus\Omega_{\varepsilon^{\gamma}})\setminus(\Omega^{0}_{R_0}\setminus\Omega^{0}_{\varepsilon^{\gamma}})$ is $\varepsilon$, using Lemma \ref{lem_49},
\begin{align*}
\int_{(\Omega_{R_{0}}\setminus\Omega_{\varepsilon^{\gamma}})\setminus(\Omega^{0}_{R_0}\setminus\Omega^{0}_{\varepsilon^{\gamma}})}|\nabla{v}_{1}|^{2}
&\leq\,C\varepsilon\int_{\varepsilon^{\gamma}<|x'|<R_0}\frac{dx'}{|x'|^{2m}}\\
&\leq\,C\varepsilon^{1+(n-2m-1)\gamma}\leq\,C\varepsilon^{\frac{1}{2}+\frac{n-1}{4m}}\leq\,CE_{n,m}(\varepsilon).
\end{align*}
For the second and third terms, for any $\varepsilon^{\gamma}\leq|z'|\leq{R}_{0}$, $\gamma=\frac{1}{4m}$, if $\partial{D}_{1}^{0}$ and $\partial{D}_{2}^{0}$ are of $C^{2,\alpha}$, then
by Lemma \ref{lem_49},  
\begin{align*}
\int_{\Omega^{0}_{R_0}\setminus\Omega^{0}_{\varepsilon^{\gamma}}}|\nabla(v_{1}-v_{1}^{0})|^{2}\leq&\,C\varepsilon^{1/2}\int_{\Omega^{0}_{R_0}\setminus\Omega^{0}_{\varepsilon^{\gamma}}}|x'|^{-2m}dx'dx_{n}\\
\leq&\,C\varepsilon^{1/2}\int_{\varepsilon^{\gamma}<|x'|<R_0}\frac{dx'}{|x'|^{m}}
\leq\,C\varepsilon^{1/4}E_{n,m}(\varepsilon),
\end{align*}
and
$$\left|2\int_{\Omega^{0}_{R_0}\setminus\Omega^{0}_{\varepsilon^{\gamma}}}\nabla{v}_{1}^{0}\cdot\nabla(v_{1}-v_{1}^{0})\right|\leq\,C\varepsilon^{1/4}\int_{\varepsilon^{\gamma}<|x'|<R_0}\frac{dx'}{|x'|^{m}}\leq\,CE_{n,m}(\varepsilon).$$
Thus, \eqref{claim2} holds, so does \eqref{equ4.12}  with \eqref{claim1}.

\noindent{\bf STEP 2.} Next, we use the explicit functions $\bar{v}_{1}^{0}$ and $\bar{v}_{1}$ to approximate ${v}_{1}^{0}$ and $v_{1}$, respectively. 

Denote
$$M_{2}:=2\int_{\Omega^{0}_{R_{0}}}\nabla\bar{v}_{1}^{0}\cdot\nabla(v^{0}_{1}-\bar{v}_{1}^{0})
+\int_{\Omega^{0}_{R_{0}}}\left(|\nabla(v^{0}_{1}-\bar{v}_{1}^{0})|^{2}+|\partial_{x'}\bar{v}_{1}^{0}|^{2}\right),$$
which is a constant, depending on $R_{0}$ but not on $\varepsilon$. Using \eqref{nablau_starbar} and \eqref{nablaw_0star}, a similar argument as in Step 1 yields
\begin{align*}
\mathrm{I}_{2}^{(0)}
=&\int_{\Omega^{0}_{R_{0}}\setminus\Omega^{0}_{\varepsilon^{\gamma}}}|\nabla\bar{v}_{1}^{0}|^{2}
+2\int_{\Omega^{0}_{R_{0}}\setminus\Omega^{0}_{\varepsilon^{\gamma}}}\nabla\bar{v}_{1}^{0}\cdot\nabla(v^{0}_{1}-\bar{v}_{1}^{0})
+\int_{\Omega^{0}_{R_{0}}\setminus\Omega^{0}_{\varepsilon^{\gamma}}}|\nabla(v^{0}_{1}-\bar{v}_{1}^{0})|^{2}\\
=&\int_{\Omega^{0}_{R_{0}}\setminus\Omega^{0}_{\varepsilon^{\gamma}}}|\partial_{x_{n}}\bar{v}_{1}^{0}|^{2}
+ M_{2}+O(\varepsilon^{\frac{n-1}{4m}}).
\end{align*}

For term $\mathrm{II}$ in \eqref{energy_v1},
\begin{align}\label{nabla_v1_part3}
\mathrm{II}=\int_{\Omega_{\varepsilon^{\gamma}}}|\nabla{v}_{1}|^{2}=
&\int_{\Omega_{\varepsilon^{\gamma}}}|\nabla\bar{v}_{1}|^{2}
+2\int_{\Omega_{\varepsilon^{\gamma}}}\nabla\bar{v}_{1}\cdot\nabla(v_{1}-\bar{v}_{1})
+\int_{\Omega_{\varepsilon^{\gamma}}}|\nabla(v_{1}-\bar{v}_{1})|^{2}.
\end{align}
By Proposition \ref{prop_v1}, we have
$$2\int_{\Omega_{\varepsilon^{\gamma}}}\nabla\bar{v}_{1}\cdot\nabla(v_{1}-\bar{v}_{1})
+\int_{\Omega_{\varepsilon^{\gamma}}}|\nabla(v_{1}-\bar{v}_{1})|^{2}=O\left(\varepsilon^{\frac{n-1}{4m}}\right).
$$
Recalling the assumption \eqref{h1h24}-\eqref{h1h25} and \eqref{ubar}, we have
\begin{equation}\label{thm1.8_equ1}
\left|\partial_{x'}\bar{v}_{1}(x)\right|\leq\frac{C|x'|^{m-1}}{\varepsilon+|x'|^{m}},\quad
\partial_{x_{n}}\bar{v}_{1}(x)=\frac{1}{\varepsilon+(h_{1}-h_{2})(x')},\qquad~x\in\Omega_{\varepsilon^\gamma}.
\end{equation}
Therefore
$$\int_{\Omega_{\varepsilon^{\gamma}}}|\partial_{x'}\bar{v}_{1}|^{2} \le C \int_{|x'|<\varepsilon^\gamma} \frac{|x'|^{2m-2}}{\varepsilon + |x'|^{m}} \,dx'  \le C \int_{|x'|<\varepsilon^\gamma} \,|x'|^{m-2}dx' = O\left(\varepsilon^{\frac{n+m-3}{4m}}\right).$$
Since $m\geq2$, then $n+m-3\geq\,n-1$. Hence, it follows from \eqref{nabla_v1_part3} and $m\geq2$ that
\begin{align*}
\mathrm{II}=&\int_{\Omega_{\varepsilon^{\gamma}}}|\partial_{x_{n}}\bar{v}_{1}|^{2} + O\left(\varepsilon^{\frac{n-1}{4m}}\right).
\end{align*}

Now combining Step 1 with the above, using $\gamma=1/(4m)$, we obtain
\begin{align}\label{energy_v1_2}
\int_{\Omega}|\nabla{v}_{1}|^{2} =&\int_{\Omega_{\varepsilon^{\gamma}}}|\partial_{x_{n}}\bar{v}_{1}|^{2} +\int_{\Omega^{0}_{R_{0}}\setminus\Omega^{0}_{\varepsilon^{\gamma}}}|\partial_{x_{n}}\bar{v}_{1}^{0}|^{2}
+ M_{1}+M_{2}+O\Big(E_{n}^{m}(\varepsilon)\Big).
\end{align}
\noindent{\bf STEP 3.} Next, we will calculate the first two terms in the right hand side of \eqref{energy_v1_2}. It follows from \eqref{nablau_starbar} and \eqref{thm1.8_equ1} that 
\begin{align}\label{equ3.17}
&\int_{\Omega_{\varepsilon^{\gamma}}}|\partial_{x_{n}}\bar{v}_{1}|^{2}+\int_{\Omega^{0}_{R_{0}}\setminus\Omega^{0}_{\varepsilon^{\gamma}}}|\partial_{x_{n}}\bar{v}_{1}^{0}|^{2}\nonumber\\=&\int_{R_0>|x'|>\varepsilon^{\gamma}}\frac{dx'}{(h_{1}-h_{2})(x')}+\int_{|x'|<\varepsilon^{\gamma}}\frac{dx'}{\varepsilon+(h_{1}-h_{2})(x')}\nonumber\\
=&\int_{\varepsilon^\gamma < |x'| < R_0} \frac{dx'}{\lambda|x'|^{m}} +  \int_{|x'| < \varepsilon^\gamma} \frac{dx'}{\varepsilon + \lambda|x'|^{m}}\nonumber\\
= &\int_{|x'| <R_0} \frac{dx'}{\varepsilon + \lambda|x'|^{m}}+ O\left(\varepsilon^{\frac{1}{2}+\frac{n-1}{4m}}  \right),
\end{align}
we here used that
\begin{align*}
\left|\int_{\varepsilon^\gamma < |x'| < R_0}\Big( \frac{1}{\lambda|x'|^{m}}-\frac{1}{\varepsilon+ \lambda|x'|^{m}}\Big)  \,dx' \right|
 \le& C\varepsilon \int_{\varepsilon^\gamma < |x'| < R_0} \frac{dx'}{|x'|^{2m}}\\
 \leq&\,C\varepsilon^{1+(n-2m-1)\gamma}\leq\,C\varepsilon^{\frac{1}{2}+\frac{n-1}{4m}}.
\end{align*}

Finally, we calculate the first term in the line of \eqref{equ3.17}. 

(i) For $n=2$, we have
\begin{align*}
2 \int_0^{R_{0}}  \frac{dx_1}{\varepsilon +\lambda\,x_{1}^{m} }
= \frac{\mathcal{L}_{m,2}}{\varepsilon^{1-\frac{1}{m}}\lambda^{\frac{1}{m}}}+M_{3}^{(1)} +O\left( \varepsilon^{\frac{1}{4m}} \right),
\end{align*}
where 
$$\mathcal{L}_{m,2}:=\int_{0}^{+\infty}\frac{1}{1+y^{m}}dy,\quad\,M_{3}^{(1)}:=\frac{2}{\lambda}\frac{m-1}{R_0^{m-1}},\quad\,m\geq2.$$
Therefore, from \eqref{energy_v1_2},
$$\int_{\Omega}|\nabla{v}_{1}|^{2} = \frac{\mathcal{L}_{m,2}}{\varepsilon^{1-\frac{1}{m}}\lambda^{\frac{1}{m}}}+M+O\left(\varepsilon^{\frac{1}{4m}}\right),\quad\,M=M_{1}+M_{2}+M_{3}^{(1)}.$$

(ii) For $n\geq3$, $m=n-1$, for the first term of \eqref{equ3.17}, \begin{align}\label{n=3part2}
\begin{split}
\int_{|x'| <R_0} \frac{dx'}{\varepsilon + \lambda|x'|^{m}} =&\frac{\omega_{n-1}}{\lambda}\int_{0}^{R_{0}(\frac{\lambda}{\varepsilon})^{1/m}}\frac{r^{m-1}dr}{1+r^{m}}\\
=&\frac{\mathcal{L}_{m,n}}{\lambda|\log\varepsilon|} +M_{3}^{(2)}+ O(\varepsilon^{\frac{1}{m}}),
\end{split}
\end{align} 
where
$$\mathcal{L}_{m,n}:=\frac{\omega_{n-1}}{m},\quad\,M_{3}^{(2)}:=\frac{\omega_{n-1}}{\lambda}(\log R_{0}+\frac{1}{m}\log\lambda).$$
Therefore, from \eqref{energy_v1_2},
$$\int_{\Omega}|\nabla{v}_{1}|^{2} =\frac{\mathcal{L}_{m,n}}{\lambda|\log\varepsilon|}+M+ O(\varepsilon^{\frac{1}{n-1}})+O\Big(E_{n}^{n-1}(\varepsilon)\Big),$$
where
$$M=M_{1}+M_{2}+M_{3}^{(2)}.$$

(iii) For $n\geq3$, $m>n-1$
\begin{align*}\label{n=3part2}
\int_{|x'| <R_0} \frac{dx'}{\varepsilon + \lambda|x'|^{m}} =&\frac{\omega_{n-1}}{\lambda^{\frac{n-1}{m}}\varepsilon^{1-\frac{n-1}{m}}}\int_{0}^{R_{0}(\frac{\lambda}{\varepsilon})^{1/m}}\frac{r^{n-2}dr}{1+r^{m}}\\
=&\frac{\mathcal{L}_{m,n}}{\lambda^{\frac{n-1}{m}}\varepsilon^{1-\frac{n-1}{m}}} +M_{3}^{(3)}+ O(\varepsilon^{2-\frac{n-1}{m}}),
\end{align*} 
where
$$\mathcal{L}_{m,n}:=\omega_{n-1}\int_{0}^{+\infty}\frac{r^{n-2}dr}{1+r^{m}},\quad\,M_{3}^{(3)}:=\frac{\omega_{n-1}}{\lambda}R_{0}^{n-1-m}.$$
Therefore, from \eqref{energy_v1_2},
$$\int_{\Omega}|\nabla{v}_{1}|^{2} =\frac{\mathcal{L}_{m,n}}{\lambda^{\frac{n-1}{m}}\varepsilon^{1-\frac{n-1}{m}}}+M+ O(\varepsilon^{\frac{1}{4m}}),$$
where
$$M=M_{1}+M_{2}+M_{3}^{(3)}.$$

It is not difficult to prove that these $M$ are some constants independent of $R_{0}$. If not, suppose that there exist $M(R_{0})$ and $M(\tilde{R}_{0})$, both independent of $\varepsilon$, such that \eqref{asym2} holds, then
$$M(R_{0})-M(\tilde{R}_{0})\rightarrow0,\quad\mbox{as}~~\varepsilon\rightarrow0,$$
which implies that $M(R_{0})=M(\tilde{R}_{0})$.
\end{proof}

\bigskip

\section{The proof of Theorem \ref{thm2} and Proposition \ref{prop2}}\label{Sec_prop2}

\subsection{Estimates for $|\nabla{v}_{b}|$}
\begin{proof}[Proof of Theorem \ref{thm2}] First, by the trace theorem, we have $|C_{2}|\leq\,C$. Recall that $v_{b}$ satisfies that
\begin{equation}\label{eqv2}
\begin{cases}
\Delta({v}_{b}-C_{2})=0&\mbox{in}~\Omega,\\
v_{b}-C_{2}=0&\mbox{on}~\partial{D}_{1}\cup\partial{D}_{2},\\
v_{b}-C_{2}=\varphi(x)-C_{2}&\mbox{on}~\partial{D}.
\end{cases}
\end{equation}
For any $0<t<s<R_{0}$, we introduce a cutoff function $\eta\in{C}^{\infty}(\Omega_{R_{0}})$ satisfying $0\leq\eta\leq1$, $\eta=1$ in $\Omega_{t}(z')$, $\eta=0$ in $\Omega_{R_{0}}\setminus\Omega_{s}(z')$, and $|\nabla\eta|\leq\frac{2}{s-t}$. Multiplying $\eta^{2}(v_{b}-C_{2})$ on the both sides of the equation in \eqref{eqv2} and applying the integration by parts, we have
$$\int_{\Omega_{s}(z')}|\nabla(v_{b}-C_{2})|^{2}\eta^{2}dx\leq\frac{C}{(s-t)^{2}}\int_{\Omega_{s}(z')}|v_{b}-C_{2}|^{2}dx.$$
Since $v_{b}-C_{2}=0$ on $\partial{D}_{2}$, by H\"older inequality, we have
$$\int_{\Omega_{s}(z')}|v_{b}-C_{2}|^{2}\leq\,C\delta(z')^{2}\int_{\Omega_{s}(z')}|\nabla{v}_{b}|^{2}dx.$$
Thus, we have
\begin{equation}\label{energy}
\int_{\Omega_{t}(z')}|\nabla{v}_{b}|^{2}dx\leq\,C\left(\frac{\delta(z')}{s-t}\right)^{2}\int_{\Omega_{s}(z')}|\nabla{v}_{b}|^{2}dx
\end{equation}
For simplicity, denote
$$F(t):=\int_{\Omega_{t}(z')}|\nabla{v}_{b}|^{2}dx,$$
then \eqref{energy} can be written as
$$F(t)\leq\,\left(\frac{C_{0}\delta(z')}{s-t}\right)^{2}F(s),$$
here we fix the universal constant $C_{0}$. Let $t_{0}=\delta$, $t_{i+1}=t_{i}+2C_{0}\delta$, then we have the following iteration formula
$$F(t_{i})\leq\frac{1}{4}F(t_{i+1}).$$
After $k=\left[\frac{\delta(z')^{1/m}}{2C_{0}\delta(z')}\right]$ times, we have
$$F(t_{0})\leq(\frac{1}{4})^{k}\int_{\Omega_{|z'|}(z')}|\nabla{v}_{b}|^{2}dx\leq\,C(\frac{1}{4})^{\left[\frac{1}{2C_{0}\delta(z')^{1-1/m}}\right]}.$$
So that
$$\int_{\Omega_{\delta(z')}(z')}|\nabla{v}_{b}|^{2}dx\leq\,C\exp(-\frac{1}{2C_{0}\delta(z')^{1-1/m}}).$$
A similar procedure as Step 2 in the proof of Proposition \ref{prop_v1} yields \eqref{thm2.5_equ}.
\end{proof}

\subsection{Proof of Propostion \ref{prop2}}

We recall the decomposition as in \cite{bly1}
\begin{equation}\label{decomposition_u3}
u(x)=C_{1}v_{1}(x)+C_{2}v_{2}(x)+v_{0}(x),\quad\mbox{in}~~\Omega,
\end{equation}
where $v_{1}$ is defined in \eqref{equv1}, $v_{2}$ and $v_{0}$ are, respectively, the solutions of
\begin{equation}\label{equv2}
\begin{cases}
\Delta{v}_{2}=0&\mbox{in}~\Omega,\\
v_{2}=1&\mbox{on}~\partial{D}_{2},\\
v_{2}=0&\mbox{on}~\partial{D}_{1}\cup\partial{D},\end{cases}
\quad\mbox{and}\quad
\begin{cases}
\Delta{v}_{0}=0&\mbox{in}~\Omega,\\
v_{0}=0&\mbox{on}~\partial{D}_{1}\cup\partial{D}_{2},\\
v_{0}=\varphi(x)&\mbox{on}~\partial{D}.
\end{cases}
\end{equation}
Then $(v_{1}+v_{2})$ satisfies
\begin{equation}\label{v1+v2}
\begin{cases}
\Delta (v_{1}+v_{2})=0,&\mbox{in}~\Omega,\\
v_{1}+v_{2}=1,&\mbox{on}~\partial{D}_{1}\cup\partial{D}_{2},\\
v_{1}+v_{2}=0,&\mbox{on}~\partial{D}.
\end{cases}
\end{equation}

We decompose $u_{0}$ into
$$u_{0}=C_{0}u_{0}^{1}+u_{0}^{0},\quad\mbox{in}~\Omega^{0},$$
where $u_{0}^{1},u_{0}^{0}\in{C}^{1}(\overline{\Omega})$ are, respectively, the solutions of
\begin{equation}\label{u0}
\begin{cases}
\Delta u_{0}^{1}=0,&\mbox{in}~\Omega^{0},\\
u_{0}^{1}=1,&\mbox{on}~\partial{D}_{1}^{0}\cup\partial{D}_{2}^{0},\\
u_{0}^{1}=0,&\mbox{on}~\partial{D},
\end{cases}\quad
\mbox{and}\quad
\begin{cases}
\Delta u_{0}^{0}=0,&\mbox{in}~\Omega^{0},\\
u_{0}^{0}=0,&\mbox{on}~\partial{D}_{1}^{0}\cup\partial{D}_{2}^{0},\\
u_{0}^{0}=\varphi,&\mbox{on}~\partial{D}.
\end{cases}
\end{equation}

To prove Proposition \ref{prop2}, we need the following lemmas. 

\begin{lemma}\label{lem5.1}
\begin{equation}\label{convergence_v12}
\left|\int_{\partial{D}_{i}}\frac{\partial(v_{1}+v_{2})}{\partial\nu^{-}}-\int_{\partial{D}_{i}^{0}}\frac{\partial{u}_{0}^{1}}{\partial\nu^{-}}\right|\leq\,C\varepsilon^{1^{-}},\quad\,i=1,2,
\end{equation}
and
\begin{equation}\label{convergence_v0}
\left|\int_{\partial{D}_{i}}\frac{\partial{v}_{0}}{\partial\nu^{-}}-\int_{\partial{D}_{i}^{0}}\frac{\partial{u}_{0}^{0}}{\partial\nu^{-}}\right|\leq\,C\varepsilon^{1^{-}},\quad\,i=1,2,
\end{equation}
where $\varepsilon^{1^{-}}$ means $\varepsilon^{1-\eta}$ for any small positive constant $\eta$.
\end{lemma}

\begin{proof} We only prove \eqref{convergence_v12} with $i=1$ for instance, the others are the same.
It follows from Theorem \ref{thm2} with $\varphi(x)=C_{2}=1$ that
\begin{equation}\label{nabla(v_{1}+v_{2})}
|\nabla(v_{1}+v_{2})|\leq\,C,\quad\mbox{in}~~\Omega.
\end{equation}
Because of the same reason,
\begin{equation}\label{nabla{v}^{*1}}
|\nabla{u}_{0}^{1}|\leq\,C,\quad\mbox{in}~~\Omega^{0}.
\end{equation}

Letting
$$\phi_{1}:=(v_{1}+v_{2})-u^{1}_{0},$$
then $\Delta\phi_{1}=0$ in $V=D\setminus\overline{D_{1}\cup{D}_{1}^{0}\cup{D}_{2}}$, and $\phi_{1}=0$ on $\partial{D}$. It is obvious that  $(v_{1}+v_{2})=u_{0}^{1}=1$ on $\partial{D}_{2}$, that is, $\phi_{1}=0$  on $\partial{D}_{2}$. 
On $\partial{D}_{1}^{0}\setminus{D}_{1}$, by using mean value theorem and \eqref{nabla(v_{1}+v_{2})}, we have
\begin{align*}
|\phi_{1}|\Big|_{\partial{D}_{1}^{0}\setminus{D}_{1}}&=|(v_{1}+v_{2})-{u}_{0}^{1}|\Big|_{\partial{D}_{1}^{0}\setminus{D}_{1}}
=|(v_{1}+v_{2})-1|\Big|_{\partial{D}_{1}^{0}\setminus{D}_{1}}\\
&=|(v_{1}+v_{2})(x',x_{d})-(v_{1}+v_{2})(x',x_{d}+\varepsilon)|\Big|_{\partial{D}_{1}^{0}\setminus{D}_{1}}\\
&\leq|\nabla (v_{1}+v_{2})(\xi)|\varepsilon\leq\,C\varepsilon,
\end{align*}
for some $\xi\in\Omega$; similarly, using \eqref{nabla{v}^{*1}},
\begin{align*}
|\phi_{1}|\Big|_{\partial{D}_{1}\setminus{D}_{1}^{0}}&=|(v_{1}+v_{2})-{u}_{0}^{1}|\Big|_{\partial{D}_{1}\setminus{D}_{1}^{0}}=|1-u_{0}^{1}|\Big|_{\partial{D}_{1}\setminus{D}_{1}^{0}}\\
&=|{u}_{0}^{1}(x',x_{d}-\varepsilon)-{u}_{0}^{1}(x',x_{d})|\Big|_{\partial{D}_{1}\setminus{D}_{1}^{0}}=|\nabla {u}_{0}^{1}(\xi)|\varepsilon\leq\,C\varepsilon,
\end{align*}
for some another $\xi\in\Omega^{0}$. Applying the maximum principle to $\phi_{1}$ on $V$, we have
\begin{equation}\label{phi1_bdd}
|\phi_{1}|\leq\,C\varepsilon,\quad\mbox{on}~~V.
\end{equation}

Denote
$$\Omega^{+}:=V\cap\{x\in\Omega|x_{d}>0\},\quad~~(\partial{D})^{+}:=\{x\in\partial{D}|x_{d}>0\},~~\mbox{and}~~\gamma=\{x_{d}=0\}\cap\Omega.$$
Since $(v_{1}+v_{2})$ and $u_{0}^{1}$ are harmonic in $\Omega^{+}\setminus{D}_{1}$ and $\Omega^{+}\setminus{D}_{1}^{0}$, respectively, by using integration by parts, 
$$0=\int_{\Omega^{+}\setminus{D}_{1}}\Delta(v_{1}+v_{2})=\int_{\partial{D}_{1}}\frac{\partial(v_{1}+v_{2})}{\partial\nu^{-}}+\int_{(\partial{D})^{+}}\frac{\partial(v_{1}+v_{2})}{\partial\nu}+\int_{\gamma}\frac{\partial(v_{1}+v_{2})}{\partial\nu},$$
and
$$0=\int_{\Omega^{+}\setminus{D}_{1}^{0}}\Delta{u}_{0}^{1}=\int_{\partial{D}_{1}^{0}}\frac{\partial{u}_{0}^{1}}{\partial\nu^{-}}+\int_{(\partial{D})^{+}}\frac{\partial{u}_{0}^{1}}{\partial\nu}+\int_{\gamma}\frac{\partial{u}_{0}^{1}}{\partial\nu}.$$
Thus,
$$\int_{\partial{D}_{1}^{0}}\frac{\partial{u}_{0}^{1}}{\partial\nu^{-}}-\int_{\partial{D}_{1}}\frac{\partial(v_{1}+v_{2})}{\partial\nu^{-}}=\int_{(\partial{D})^{+}}\frac{\partial\phi_{1}}{\partial\nu}+\int_{\gamma}\frac{\partial\phi_{1}}{\partial\nu}.$$

First, using the standard boundary gradient estimates for $\phi_{1}$ and \eqref{phi1_bdd}, we have
$$\Big|\int_{(\partial{D})^{+}}\frac{\partial\phi_{1}}{\partial\nu}\Big|\leq\,C\varepsilon.$$
Divide $\gamma$ into three pieces: $\gamma=\gamma_{1}\cup\gamma_{2}\cup\gamma_{3}$, where
$$\gamma_{1}:=\{(x',0)~|~|x'|\leq\,\frac{A}{2|\log\varepsilon|}\},\quad\gamma_{2}:=\{(x',0)~|~\frac{A}{2|\log\varepsilon|}<|x'|<R_{0}\},$$
$$\gamma_{3}:=\gamma\setminus(\gamma_{1}\cup\gamma_{2}),$$
the constant $A$ is determined in \eqref{thm2.5_equ}. Write
$$\int_{\gamma}\frac{\partial\phi_{1}}{\partial\nu}=\int_{\gamma_{1}}+\int_{\gamma_{2}}+\int_{\gamma_{3}}\frac{\partial\phi_{1}}{\partial\nu}:=\mathrm{I}+\mathrm{II}+\mathrm{III}.$$
For $(y',0)\in\gamma_{1}$, 
by Theorem \ref{thm2}, 
$$|\nabla (v_{1}+v_{2})|,|\nabla u_{0}^{1}|\leq\,C\exp(-\frac{A}{(\varepsilon+|x'|^{m})^{1-\frac{1}{m}}}),\quad\mbox {in}~\Omega_{R_{0}}.$$ Hence
$$|\mathrm{I}|\leq\,C\varepsilon.$$
For $(y',0)\in\gamma_{2}$, there exists a $r>\frac{1}{C}|y'|^{m}$ for some $C>1$ such that $B_{r}(y',0)\subset{V}$. It then follows from the standard gradient estimates for harmonic function and \eqref{phi1_bdd} that
$$|\nabla\phi_{1}(y',0)|\leq\frac{C\varepsilon}{|y'|^{m}},$$
and
$$|\mathrm{II}|\leq\,C\varepsilon\int_{\frac{A}{2|\log\varepsilon|}<|y'|<R_{0}}\frac{1}{|y'|^{m}}dS\leq\,C
\begin{cases}
\varepsilon|\log\varepsilon|^{m-n+1}&\mbox{for}~m>n-1,n\geq3,\\
\varepsilon\log|\log\varepsilon|&\mbox{for}~m=n-1, n\geq3,\\
\varepsilon|\log\varepsilon|^{m-1}&\mbox{for}~m\geq2,n=2.
\end{cases}$$
For $(y',0)\in\gamma_{3}$, there is a universal constant $r>0$ such that $B_{r}(x)\subset{V}$ for all $x\in\gamma_{3}$. So we have from \eqref{phi_bdd} that for any $x\in\gamma_{3}$,
$$|\nabla\phi_{1}|\leq\frac{C\varepsilon}{r}\leq\,C\varepsilon,$$
and
$$|\mathrm{III}|\leq\,C\varepsilon.$$
Thus, we have \eqref{convergence_v12} with $i=1$. 
\end{proof}

\begin{lemma}\label{lem6.2}
 Let $C_{1}$ and $C_{2}$ be defined in \eqref{equinfty2} and $C_{0}$ be in \eqref{equ_u0}. We have
\begin{equation}\label{C*}
\left|\frac{C_{1}+C_{2}}{2}-C_{0}\right|\leq\,C\Big(\rho^{m}_{n}(\varepsilon)\Big).
\end{equation}
As a consequence, combining it with \eqref{C1-C2_2}, we have
\begin{equation}\label{C*-C2}
\left|C_{i}-C_{0}\right|\leq\left|C_{i}-\frac{C_{1}+C_{2}}{2}\right|+\left|\frac{C_{1}+C_{2}}{2}-C_{0}\right|\leq\,C\Big(\rho^{m}_{n}(\varepsilon)\Big),\quad\,i=1,2.
\end{equation}
\end{lemma}

\begin{proof}
In view of the decomposition \eqref{decomposition_u3}, the third line of \eqref{equinfty2}, we have
\begin{align}\label{system_C1C2}
C_{1}\int_{\partial{D}_{i}}\frac{\partial{v}_{1}}{\partial\nu^{-}}
+C_{2}\int_{\partial{D}_{i}}\frac{\partial{v}_{2}}{\partial\nu^{-}}+\int_{\partial{D}_{i}}\frac{\partial{v}_{0}}{\partial\nu^{-}}=0,\quad\,i=1,2.
\end{align}
Let
$$a_{ij}=\int_{\partial{D}_{i}}\frac{\partial{v}_{j}}{\partial\nu^{-}},\quad\,b_{i}=-\int_{\partial{D}_{i}}\frac{\partial{v}_{0}}{\partial\nu^{-}}.$$
That is,
$$\begin{cases}
a_{11}C_{1}+a_{12}C_{2}=b_{1},\\
a_{21}C_{1}+a_{22}C_{2}=b_{2}.
\end{cases}$$
So that
$$(a_{11}+a_{21})C_{1}+(a_{12}+a_{22})C_{2}=b_{1}+b_{2}.$$
Since $a_{12}=a_{21}$, it follows that
$$(a_{11}+a_{21})(C_{1}+C_{2})+(a_{22}-a_{11})C_{2}=b_{1}+b_{2}.$$
Similarly,
$$(a_{12}+a_{22})(C_{1}+C_{2})-(a_{22}-a_{11})C_{1}=b_{1}+b_{2}.$$
Adding these two equations together and dividing by two yields
\begin{equation*}
(a_{11}+a_{21}+a_{12}+a_{22})\frac{C_{1}+C_{2}}{2}+(a_{22}-a_{11})\frac{(C_{2}-C_{1})}{2}=b_{1}+b_{2}.
\end{equation*}
That is,
\begin{align}\label{C1+C2/2}
&\left(\int_{\partial{D}_{1}}\frac{\partial(v_{1}+v_{2})}{\partial\nu^{-}}+\int_{\partial{D}_{2}}\frac{\partial(v_{1}+v_{2})}{\partial\nu^{-}}\right)\frac{C_{1}+C_{2}}{2}\nonumber\\
=&\left(-\int_{\partial{D}_{1}}\frac{\partial{v}_{0}}{\partial\nu^{-}}-\int_{\partial{D}_{2}}\frac{\partial{v}_{0}}{\partial\nu^{-}}\right)+\left(\int_{\Omega}|\nabla{v}_{1}|^{2}-\int_{\Omega}|\nabla{v}_{2}|^{2}\right)\frac{(C_{2}-C_{1})}{2}.
\end{align}

Recalling that $\bar{v}_{2}=1-\bar{v}_{1}$ in $\Omega_{R}$, we have $|\nabla\bar{v}_{1}|=|\nabla\bar{v}_{2}|$. By \eqref{nabla_w_in} and \eqref{nabla_v1_in},
\begin{align*}
&\left|\int_{\Omega_{R}}|\nabla{v}_{1}|^{2}-\int_{\Omega_{R}}|\nabla{v}_{2}|^{2}\right|\\
=&\left|\int_{\Omega_{R}}|\nabla({v}_{1}-\bar{v}_{1})+\nabla\bar{v}_{1}|^{2}-\int_{\Omega_{R}}|\nabla({v}_{2}-\bar{v}_{2})+\nabla\bar{v}_{2}|^{2}\right|\\
\leq&\sum_{i=1}^{2}\left|\int_{\Omega_{R}}|\nabla({v}_{i}-\bar{v}_{i})|^{2}+2\nabla\bar{v}_{i}\nabla({v}_{i}-\bar{v}_{i})\right|\leq\,C.
\end{align*}
Hence,
\begin{align*}
&\left|\int_{\Omega}|\nabla{v}_{2}|^{2}-\int_{\Omega}|\nabla{v}_{1}|^{2}\right|\leq\,C.
\end{align*}
By using Lemma \ref{lem5.1} and \eqref{C1-C2_2}, \eqref{C1+C2/2} can be written as
\begin{align}\label{lem5.2_equ1}
&\left(\int_{\partial{D}_{1}^{0}}\frac{\partial{u}_{0}^{1}}{\partial\nu^{-}}+\int_{\partial{D}_{2}^{0}}\frac{\partial{u}_{0}^{1}}{\partial\nu^{-}}+O(\varepsilon^{1^{-}})\right)\frac{C_{1}+C_{2}}{2}\nonumber\\
=&-\int_{\partial{D}_{1}^{0}}\frac{\partial{u}_{0}^{0}}{\partial\nu^{-}}-\int_{\partial{D}_{2}^{0}}\frac{\partial{u}_{0}^{0}}{\partial\nu^{-}}+O(\varepsilon^{1^{-}})+O(\rho^{m}_{n}(\varepsilon)).
\end{align}

On the other hand, from the third line of \eqref{equ_u0}, we have
\begin{equation}\label{equ_C*}
C_{0}\left(\int_{\partial{D}_{1}^{0}}\frac{\partial{u}_{0}^{1}}{\partial\nu^{-}}+\int_{\partial{D}_{2}^{0}}\frac{\partial{u}_{0}^{1}}{\partial\nu^{-}}\right)+\left(\int_{\partial{D}_{1}^{0}}\frac{\partial{u}_{0}^{0}}{\partial\nu^{-}}+\int_{\partial{D}_{2}^{0}}\frac{\partial{u}_{0}^{0}}{\partial\nu^{-}}\right)
=0.
\end{equation}
Comparing it with \eqref{lem5.2_equ1}, and in view of $|C_{i}|\leq\,C$, we have
$$\left(\int_{\partial{D}_{1}^{0}}\frac{\partial{u}_{0}^{1}}{\partial\nu^{-}}+\int_{\partial{D}_{2}^{0}}\frac{\partial{u}_{0}^{1}}{\partial\nu^{-}}\right)\Big(\frac{C_{1}+C_{2}}{2}-C_{0}\Big)=O\Big(\rho^{m}_{n}(\varepsilon)\Big)$$
Using the integration by parts and recalling the definition of $u_{0}^{1}$, we have
$$0<\int_{\partial{D}_{1}^{0}}\frac{\partial{u}_{0}^{1}}{\partial\nu^{-}}+\int_{\partial{D}_{2}^{0}}\frac{\partial{u}_{0}^{1}}{\partial\nu^{-}}=\int_{\Omega^{0}}|\nabla {u}_{0}^{1}|^{2}\leq\,C.$$
The proof of \eqref{C*} is finished.
\end{proof}

\begin{proof}[Proof of Proposition \ref{prop2}] Let
$$\phi(x):=C_{2}-C_{0}-({v}_{b}(x)-u_{0}(x)),$$
then $\Delta\phi=0$ in $V=D\setminus\overline{D_{1}\cup{D}_{1}^{0}\cup{D}_{2}}$. It is easy to see that $\phi=0$ on $\partial{D}_{2}$ and from Lemma \ref{lem6.2}
$$|\phi|\Big|_{\partial{D}}=|C_{2}-C_{0}|\leq\,C\rho^{m}_{n}(\varepsilon).$$
On $\partial{D}_{1}^{0}\setminus{D}_{1}$, by mean value theorem, \eqref{v2_bdd} and \eqref{C*-C2}, we have
\begin{align}\label{equ4.18}
|\phi(x)|\Big|_{\partial{D}_{1}^{0}\setminus{D}_{1}}=|C_{2}-v_{b}(x)|\Big|_{\partial{D}_{1}^{0}\setminus{D}_{1}}&=|\partial_{x_{n}} v_{b}(x',\xi_{n})|\varepsilon
\leq\,C\varepsilon,
\end{align}
where $\xi_{n}\in(0,\varepsilon)$. Similarly,
\begin{align*}
|\phi|\Big|_{\partial{D}_{1}\setminus{D}_{1}^{0}}=|C_{0}-u_{0}|\Big|_{\partial{D}_{1}\setminus{D}_{1}^{0}}=|\nabla u_{0}(\xi)|\varepsilon\leq\,C\varepsilon,
\end{align*}
for some $\xi\in{D}_{1}\setminus{D}_{1}^{0}$. We now apply the maximum principle to $\phi$ on $V$, 
\begin{equation}\label{phi_bdd}
|\phi|\leq\,C\rho^{m}_{n}(\varepsilon),\quad\mbox{on}~~V.
\end{equation}

Similarly as in the proof of Lemma \ref{lem5.1}, since ${v}_{b}$ and $u_{0}$ are harmonic in $\Omega^{+}\setminus{D}_{1}$ and $\Omega^{+}\setminus{D}_{1}^{0}$, respectively, by using integration by parts, we have
$$\mathcal{B}_{\varepsilon}[\varphi]=-\int_{\partial{D}_{1}}\frac{\partial{v}_{b}}{\partial\nu^{-}}=\int_{(\partial{D})^{+}}\frac{\partial{v}_{b}}{\partial\nu}+\int_{\gamma}\partial_{{x}_{n}}{v}_{b},$$
and
$$\mathcal{B}_{0}[\varphi]=-\int_{\partial{D}_{1}^{0}}\frac{\partial{u}_{0}}{\partial\nu^{-}}=\int_{(\partial{D})^{+}}\frac{\partial{u}_{0}}{\partial\nu}+\int_{\gamma}\partial_{{x}_{n}}{u}_{0}.$$
Thus,
$$\mathcal{B}_{0}[\varphi]-\mathcal{B}_{\varepsilon}[\varphi]=\int_{\partial{D}_{1}^{0}}\frac{\partial{u}_{0}}{\partial\nu^{-}}-\int_{\partial{D}_{1}}\frac{\partial{v}_{b}}{\partial\nu^{-}}=\int_{(\partial{D})^{+}}\frac{\partial\phi}{\partial\nu}+\int_{\gamma}\partial_{{x}_{n}}\phi.$$

First, as before, using the standard boundary gradient estimates for $\phi$ and \eqref{phi_bdd}, we have
$$\Big|\int_{(\partial{D})^{+}}\frac{\partial\phi}{\partial\nu}\Big|\leq\,C\rho^{m}_{n}(\varepsilon).$$
Next, similarly in the proof of Lemma \ref{lem5.1}, we divide $\gamma$ into three pieces: $\gamma=\gamma_{1}\cup\gamma_{2}\cup\gamma_{3}$,  with a minor modification, where
$$\gamma_{1}:=\{(x',0)~|~|x'|\leq\,\varepsilon^{\frac{n-1}{m(m-n+1)}}\},\quad\gamma_{2}:=\{(x',0)~|~\varepsilon^{\frac{n-1}{m(m-n+1)}}<|x'|<R_{0}\},$$
$$\gamma_{3}:=\gamma\setminus(\gamma_{1}\cup\gamma_{2}).$$
Write
$$\int_{\gamma}\partial_{{x}_{n}}\phi=\int_{\gamma_{1}}+\int_{\gamma_{2}}+\int_{\gamma_{3}}\partial_{{x}_{n}}\phi:=\mathrm{I}+\mathrm{II}+\mathrm{III}.$$
As in the proof of Lemma \ref{lem5.1}, replacing \eqref{phi1_bdd} by \eqref{phi_bdd}, it is easy to see that
$$|\mathrm{III}|\leq\,C\rho^{m}_{n}(\varepsilon).$$

Now consider term $\mathrm{II}$. On $\overline{\Omega}_{R_{0}^{0}}$, since $\phi=0$ on $\partial{D}_{2}$ and $\phi=\varepsilon\partial_{x_{n}}v_{b}(x',\xi_{n})$ on $\partial{D}_{1}^{0}\setminus{D}_{1}$ from \eqref{equ4.18}, then we choose $\bar\phi=\varepsilon\partial_{x_{n}}v_{b}(x',\xi_{n})\bar{v}_{1}^{0}$ to approximate $\phi$ in $\Omega_{R_{0}^{0}}$, where $\bar{v}_{1}^{0}$ is defined in \eqref{def_barv10}. Thus, $\phi-\bar\phi=0$ on $\partial\Omega_{R_{0}}^{0}\setminus\Omega$. Let $w_{\phi}=\phi-\bar\phi$. Since $\|v_{b}\|_{C^{3}}\leq\,C$ by theoerem 1.1 in \cite{llby}, it follows from the proof of Proposition \ref{prop_v1}, we have  $\|\nabla w_{\phi}\|\leq\,C\varepsilon$. From the definition of $\bar{v}_{1}^{0}$, \eqref{def_barv10}, and \eqref{nablau_starbar}, we have, for $(y',0)\in\gamma_{2}$, 
$$|\partial_{{x}_{n}}\bar{v}_{0}^{1}(y',0)|\leq\frac{C}{|y'|^{m}},\quad\mbox{and}\quad|\partial_{{x}_{n}}\bar\phi(y',0)|\leq\frac{C\varepsilon}{|y'|^{m}}+C\varepsilon.$$
Hence
$$\Big|\int_{\gamma_{2}}\partial_{{x}_{n}}\bar\phi\Big|\leq\,\int_{\gamma_{2}}\frac{C\varepsilon}{|y'|^{m}}+C\varepsilon\leq\,C\rho^{m}_{n}(\varepsilon).$$
Together with $\Big|\int_{\gamma_{2}}\partial_{{x}_{n}}w_{\phi}\Big|\leq\,C\varepsilon$, yields  
$$|\mathrm{II}|\leq\,C\rho^{m}_{n}(\varepsilon).$$

For $(y',0)\in\gamma_{1}$, 
by Theorem \ref{thm2}, 
$$|\nabla v_{b}|,|\nabla u_{0}|\leq\,C\exp(-\frac{A}{(\varepsilon+|x'|^{m})^{1-\frac{1}{m}}}),\quad\mbox {in}~\Omega_{R}^{0}.$$ Hence
$$|\mathrm{I}|\leq\,C\varepsilon.$$
Thus, the proof of Propostion \ref{prop2} is completed.
\end{proof}

\vspace{.5cm}

\noindent{\bf{\large Acknowledgements.}} The author would like to express his gratitude to Professor Yanyan Li for his constant encouragement in this project. The author thank the anonymous referee for helpful suggestions which improved the exposition.

%%%%%%%%%%%%%%%%%%%%%%%%%%%%%%%%%%%%%%%%%%%%%%%%%%%%%%%%%%%%
%%%%%%%%%%%%%%%%%%%%%%%%%%%%%%%%%%%%%%%%%%%%%%%%%%%%%%%%%%%%%%

%

\end{document}